\documentclass[10pt,reqno,a4paper]{article} 

\usepackage{anysize}
\marginsize{3.0cm}{3.0cm}{2.2cm}{2.2cm}

 \usepackage{verbatim}
\usepackage[english]{babel}
\usepackage[centertags]{amsmath}
\usepackage{amsfonts,amssymb,amsthm}
\usepackage[svgnames]{xcolor}
\usepackage{comment}
\usepackage{hyperref} 
\usepackage{mathrsfs}
\usepackage[sans]{dsfont}
\usepackage{accents}

\let\OLDthebibliography\thebibliography
\renewcommand\thebibliography[1]{
  \OLDthebibliography{#1}
  \setlength{\parskip}{0pt}
  \setlength{\itemsep}{0pt plus 0.3ex} }

\numberwithin{equation}{section}

\theoremstyle{plain}
\newtheorem{theorem}{Theorem}[section]

\newtheorem{lemma}[theorem]{Lemma}

\theoremstyle{definition}
\newtheorem{remarkbase}[theorem]{Remark}
\newenvironment{remark}{\pushQED{\qed} \remarkbase}{\popQED\endremarkbase}

\newcommand{\R}{{\mathbb R}}

\renewcommand{\S}{{\mathbb S}}

\newcommand{\mE}{\mathcal{E}}

\newcommand{\mJ}{\mathcal{J}}

\newcommand{\mL}{\mathcal{L}}

\renewcommand{\a}{\alpha}

\newcommand{\ph}{\varphi}
\newcommand{\lm}{\lambda}

\newcommand{\Om}{\Omega}

\newcommand{\s}{\sigma}

\newcommand{\la}{\langle}
\newcommand{\ra}{\rangle}
\newcommand{\pa}{\partial}
\renewcommand{\div}{\operatorname{div}} 
\newcommand{\grad}{\nabla}

\newcommand{\curl}{\operatorname{curl}}

\newcounter{mt}

\def\maindefinitiondeclaration#1{\stepcounter{mt}\newcounter{#1}\setcounter{#1}{\arabic{mt}}}

\maindefinitiondeclaration{loro}
\maindefinitiondeclaration{defn:steady}

\newcommand{\bcb}{\begin{color}{blue}}
\newcommand{\bcr}{\begin{color}{red}}
\newcommand{\bcg}{\begin{color}{green}}
\newcommand{\ec}{\end{color}}

\newcommand{\Hess}{\grad^2} 

\title{A rigidity result for the 3D capillary liquid drop with constant vorticity}
\author{\normalsize{Pietro Baldi,  Domenico Angelo La Manna, Giuseppe La Scala}}
\date{} 

\begin{document}\maketitle 
\abstract{
We consider the free boundary problem for the Euler equations of fluid dynamics 
governing the motion of a 3D liquid drop with capillarity $\sigma_0$ and nearly spherical shape, 
under the assumption of constant vorticity $(0, 0, \a_0)$.
First we study the compatibility of the constant vorticity condition 
with the evolution in time of the system, 
showing that, for $\alpha_0 \neq 0$, 
any smooth solution with convex domain must satisfy 
a strong geometrical constraint on the shape of the fluid domain, 
and that the constant vorticity condition 
(unlike in the irrotational case $\a_0 = 0$) 
does not define an invariant set for the time evolution of the system. 
Then we focus on the time-independent solutions of the problem  
and we prove a new rigidity result: starting without assuming any symmetry condition, 
we show that, if the ratio $\alpha_0^2/\sigma_0$ is not too large, 
then any nearly spherical solution has necessarily cylindrical symmetry, 
and therefore it is the unique axisymmetric solution already known in literature, 
the fluid domain is close, but not equal, to a ball,  
more precisely it is an oblate spheroid, flattened at the poles and bulged at the equator,  
and each fluid particle moves along a horizontal, circular trajectory 
with constant angular velocity. 
To the best of our knowledge, this is the first result 
for the capillary liquid drop with constant vorticity
obtained without assuming cylindrical symmetry.}

\tableofcontents

\section{Introduction}

We consider the free boundary problem for the Euler equations 
of the fluid dynamics 
\begin{equation}\label{eq:system}
\begin{cases}
\pa_t u + \la u , \grad \ra u + \grad p = 0 
& \text{in } \Omega_t,
\\
\div u=0 & \text{in }  \Omega_t,
\\ 
p = \sigma_0 H_{\pa \Omega_t} & \text{on } \partial \Omega_t,
\\
V_t=\la u,\nu_{\pa\Om}\ra & \text{on } \partial \Omega_t,
\end{cases}
\end{equation}
with the additional constraint of constant vorticity 
\begin{equation}\label{eq:constant.curl}
\curl u= \alpha_0 e_3 \quad \text{in } \Om_t.
\end{equation}
Here $\Om_t \subset \R^3$ is a bounded open set,  
and it is the domain occupied by the fluid, 
$u$ is the fluid velocity vector field, 
$p$ is the fluid pressure, 
$\sigma_0>0$ is the capillarity constant, 
$H_{\pa \Omega_t}$ is the mean curvature of the boundary $\pa \Omega_t$ , 
$V_t$ is the normal velocity of the boundary $\pa \Om_t$, 
$\nu_{\pa \Om_t}$ is the outer unit normal to the boundary, 
$\alpha_0 \in \R$ is a vorticity parameter, 
and $e_3 = (0,0,1)$. 
The unknowns of the problem are $\Om_t, u, p$.   

In this paper we first study 
the compatibility of the constant vorticity condition 
with the evolution in time of the system, 
showing that, for $\alpha_0 \neq 0$, any smooth solution with convex domain must satisfy 
a strong geometrical constraint on the shape of the fluid domain $\Om_t$, 
and that condition \eqref{eq:constant.curl} 
(unlike in the irrotational case $\a_0 = 0$) 
does not define an invariant set for the time evolution of \eqref{eq:system}. 
Then we focus on the time-independent solutions of the problem  
and we prove a new rigidity result: starting without assuming any symmetry condition, 
we show that, if the ratio $\alpha_0^2/\sigma_0$ is not too large, 
then the solution has necessarily cylindrical symmetry, 
and therefore it is the unique axisymmetric solution already known in literature, 
the fluid domain is close, but not equal, to a ball, 
more precisely, it is an oblate spheroid, flattened at the poles and bulged at the equator,  
and each fluid particle moves along a horizontal, circular trajectory 
with constant angular velocity. 
To the best of our knowledge, this is the first result 
for the capillary liquid drop with constant vorticity 
obtained without assuming cylindrical symmetry.

In the next two subsections we introduce and state these results.

\subsection{{Compatibility of the constant vorticity condition with the time evolution}}

Before stating our results, we recall the well-known fact that the pressure 
can be expressed in terms of the other unknowns.  
Under suitable regularity assumptions, 
taking the divergence of the first equation in \eqref{eq:system} 
and using the second and third equations, one finds that 
the Laplacian $\Delta p$ of the pressure in $\Om_t$ 
and its trace $p|_{\pa \Om_t}$ at the boundary $ \pa \Om_t$ 
are determined by $\Om_t$ and $u$, 
and therefore $p$ is determined, in terms of $\Om_t,u$, 
as the unique solution of that Poisson problem with Dirichlet boundary conditions. 
For this reason, one can say that the unknowns of the problem are 
just the domain $\Om_t$ and the velocity vector field $u$. 
Existence of a classical solution to the problem \eqref{eq:system} has been proved in \cite{CS} and a continuation principle has been established in \cite{JL}.

The first result we prove is Theorem \ref{thm:Cauchy.pb}, 
which says that a smooth solution of \eqref{eq:system}, \eqref{eq:constant.curl} 
starting from a convex smooth set has a horizontal plane of symmetry, 
the boundary of the fluid domain is given by the union of two graphs, 
the velocity vector field does not depend on $x_3$ 
and its third component is zero. 
To simplify the statement, we recall the well-known fact that the barycenter velocity vector 
$c = \int_{\Om_t} u \, dx \in \R^3$ is a conserved quantity of system \eqref{eq:system}, 
that is, $c$ is independent of time along any solution, 
and therefore the fluid barycenter moves with uniform linear motion $b(t) = b_0 + c t$.
Hence the motion of the liquid drop can be decomposed into 
a trivial, constant velocity drift and a remaining component, 
which corresponds to using a coordinate frame 
where the origin of the axes and the barycenter of the fluid coincide. 
In other words, there is no loss of generality in assuming that the barycenter velocity vector 
is $\int_{\Om_t} u \, dx = 0$.

\begin{theorem} \label{thm:Cauchy.pb} 
Let $\Omega_0$ be a strictly convex set with $C^5$ boundary.
Let $u_0\in H^5(\Omega_0, \R^3)$ be a vector field with $\curl u_0 = \alpha_0 e_3$, 
with vorticity parameter $\a_0 \neq 0$,  
and barycenter velocity vector $\int_{\Om_0} u_0 \, dx = 0$. 
Let $(u, \Om_t)$ be a classical solution of the Cauchy problem for system \eqref{eq:system} 
in a time interval $(-\delta,\delta)$, $\delta > 0$, 
with initial datum $(u_0, \Om_0)$ at time $t=0$. 
If $u$ satisfies \eqref{eq:constant.curl} for $t \in (-\delta, \delta)$,
then there exists $\lm_0 \in \R$ such that the domain $\Omega_0$ is symmetric 
with respect to the plane $\{x_3=\lambda_0\}$, 
the surface $\pa \Omega_0 \cap \{x_3 > \lambda_0 \}$ is a graph over the plane domain 
$\Om_0 \cap \{ x_3 = \lm_0 \}$, 
the vector field $u_0$ is independent of the variable $x_3$, 
and it has third component $u_{0,3} = 0$.   
\end{theorem}

The proof of Theorem \ref{thm:Cauchy.pb} is in Section \ref{sec:symmetry}, and it 
is based on the Alexandrov reflection principle. 
From the equation of the time evolution of the vorticity 
and the constant vorticity condition \eqref{eq:constant.curl} 
one immediately deduces that 
\begin{equation} \label{eq:vort.equation.in.the.intro}
0 = \la \curl u , \grad \ra u = \alpha_0 \pa_{x_3} u. 
\end{equation}
Since $\a_0 \neq 0$, identity \eqref{eq:vort.equation.in.the.intro} 
implies that the velocity vector field $u$ is independent of the variable $x_3$, 
and this property of $u$ is used to obtain a symmetry property 
for the mean curvature $H_{\pa \Om_t}$ of the boundary $\pa \Om_t$. 
Then the symmetry of the curvature, by the Alexandrov reflection principle, 
leads to the symmetry of the domain $\Om_t$ itself.

\begin{remark}[\emph{Irrotational solutions have more freedom concerning the geometry of the domain $\Om_t$}]
Irrotational solutions of \eqref{eq:system} 
are solutions satisfying the constant vorticity condition \eqref{eq:constant.curl} 
with parameter $\a_0 = 0$. 
However, for $\a_0 = 0$, identity \eqref{eq:vort.equation.in.the.intro} becomes $0=0$, 
and the argument of Theorem \ref{thm:Cauchy.pb} does not apply. 
Thus, the geometric constraint on the shape of $\Om_t$ 
does not hold for irrotational solutions.  
\end{remark}

\begin{remark}[\emph{Strict convexity is not strictly necessary}] 
\label{rem:convex.hyp.in.the.intro}
The assumption of strict convexity of $\Omega_0$ is not really needed 
in the proof of Theorem \ref{thm:Cauchy.pb}, and it can be replaced by weaker geometric conditions, 
see Remark \ref{rem:convex.hyp}.
\end{remark}

\begin{remark}[\emph{Local existence for the Cauchy problem}] 
\label{rem:noia}
The existence of a classical solution $(\Om_t, u)$ of the Cauchy problem for system \eqref{eq:system}
is proved by Coutand and Shkoller, see \cite[Theorem 1.1]{CS}
for a statement in terms of $\Om_t,u$, and \cite[Theorem 1.3]{CS} 
for its Lagrangian formulation in spaces of functions defined 
on the fixed domain $[0,T] \times \Om_0$.  
The regularity assumed in Theorem \ref{thm:Cauchy.pb} 
is sufficient to apply the existence results in \cite{CS}. 
We observe that, in the proof of Theorem \ref{thm:Cauchy.pb}, 
the regularity in time of the solution 
is only used to obtain the vorticity equation $\pa_t w + \la u , \grad \ra w - \la w, \grad \ra u = 0$ 
for $w = \curl u$ at time $t=0$ 
and the identity $\pa_t (\curl u) = \curl (\pa_t u)$ at $t=0$.  
\end{remark}

A consequence of Theorem \ref{thm:Cauchy.pb} is that the only admissible initial data 
of the Cauchy problem for \eqref{eq:system} with the condition \eqref{eq:constant.curl} 
are pairs $(\Om_0, u_0)$ where 
the domain $\Om_0$ is symmetric with respect to the plane $\{ x_3 = \lm_0 \}$, for some $\lm_0 \in \R$, 
and the velocity vector field $u_0$ is independent of $x_3$. 
This is a geometric constraint that, in many cases, is not compatible with the time evolution equation, 
namely the set of pairs $(\Om_0,u_0)$ of this kind 
is not an invariant set for the dynamics of the capillary drop,  
as we show, with an explicit example, in the next theorem. 

\begin{theorem} \label{thm:not.invariant.subset}  
Let $\Omega_0 = B_1 \subset \R^3$ be the open unit ball, 
and let $u_0(x) := \frac{\a_0}{2} (-x_2, x_1, 0)$.
Let $(\Omega_t, u)$ be the unique smooth solution of the Cauchy problem for \eqref{eq:system} 
in some time interval $(-\delta, \delta)$, $\delta>0$, 
with initial data $(\Om_0, u_0)$ at time $t=0$. 
Then $\curl u$ is not constant in $(-\delta, \delta)$. 
\end{theorem}

Theorem \ref{thm:not.invariant.subset} gives an example of an initial datum $(\Om_0, u_0)$ 
where \eqref{eq:constant.curl} is satisfied, 
$\Om_0$ is symmetric with respect to the plane $\{ x_3 = 0 \}$, 
and $u_0$ depends only on $x_1, x_2$, 
and, nonetheless, the solution $(\Om_t, u)$ starting from $(\Om_0, u_0)$ 
has nonconstant vorticity in $(-\delta, \delta)$. 
In other words, we have proved that the Cauchy problem for system \eqref{eq:system} 
joint with the additional equation \eqref{eq:constant.curl} is, in general, illposed. 
The proof of Theorem \ref{thm:not.invariant.subset} is in Section \ref{sec:symmetry}.

\begin{remark}[\emph{Illposedness is a 3-dimensional phenomenon}] \label{rem:2D.3D} 
The illposedness of system \eqref{eq:system}, \eqref{eq:constant.curl} 
shown in Theorem \ref{thm:not.invariant.subset} is a genuinely 3-dimensional phenomenon:
this is due to the fact that the vorticity equation in two dimensions is 
$\pa_t w + \la u , \grad \ra w =0$
and, therefore, the constant vorticity property is preserved during the evolution.  
In fact, in dimension 2, the solution of the Cauchy problem with initial data $(\Om_0, u_0)$, 
where $\Om_0$ is the unit ball of $\R^2$ (the unit disc) and $u_0 = \frac{\a_0}{2} (-x_2, x_1)$, 
is the time-independent solution $(\Om_t,u(t)) = (\Om_0, u_0)$. 
\end{remark}

\subsection{{Rigidity result for time-independent solutions}}


The aim of the second, and main, part of the paper 
is to investigate solutions $\Om_t = \Om$, $u$ of problem \eqref{eq:system}-\eqref{eq:constant.curl} 
with 
\begin{equation} \label{time.indep.sol}
\Om, u = \text{independent of time}.
\end{equation}


In the irrotational case, that is, for $\alpha_0=0$, 
the only solutions $(\Om, u)$ with $\Omega$ of class $C^2$, time-independent and simply connected
(with $u$, in principle, possibly depending on time) 
are solutions where 
$u = 0$ and $\Om$ is a ball. 
The argument is rather simple, and we recall it here.
Assume that $(\Om,u)$ solve \eqref{eq:system}, with $\curl u = 0$ 
and $\Om$ time-independent, $C^2$ and simply connected.
Since the boundary $\pa \Om$ does not change in time, 
its normal velocity $V_t$ is zero, and, from the fourth equation in \eqref{eq:system}, 
the vector $u$ has zero normal component at the boundary $\pa \Om$. 
Hence, for any scalar function $f$, one has 
$\int_{\pa \Om} f \la u , \nu_{\pa \Om} \ra \, d\sigma = 0$, 
and, by the second equation in \eqref{eq:system} and the divergence theorem, 
$\int_\Om \la \grad f, u \ra \, dx = \int_\Om \div(fu) \, dx = 0$. 
In other words, the vector field $u$ is orthogonal in $L^2(\Om, \R^3)$ 
to any gradient vector field. 
On the other hand, since $\curl u$ is zero and $\Om$ is simply connected, 
there exists a scalar function $\Phi$, called the velocity potential, 
such that $u = \grad \Phi$. Taking $f = \Phi$ in the orthogonality property above, 
we get $0 = \int_\Om \la \grad f, u \ra \, dx = \int_\Om |\grad \Phi|^2 \, dx$, 
that is, $u = \grad \Phi = 0$. 
Since $u=0$, the first equation in \eqref{eq:system} gives $\grad p = 0$. 
Hence the pressure has the same value (possibly depending on $t$) at all points of $\Om$. 
On the other hand, $\sigma_0 H_{\pa \Om}$ is independent of time, 
and, by the third equation in \eqref{eq:system}, $p|_{\pa \Om}$ 
is independent of time too. Hence $p$ is a constant. 
The third equation in \eqref{eq:system} also says that the mean curvature of $\pa \Om$ 
is constant, and this, by Alexandrov's Theorem, implies that $\Om$ is a ball.  

When $\alpha_0\not =0$ the problem becomes more challenging, 
and the simple argument above fails. 
As we have seen above, an important difference is that 
being irrotational is a condition that is preserved along the motion 
and that gives no special geometric constraint on the shape of the fluid domain,  
while, on the contrary, the condition of nonzero constant vorticity 
implies strong geometric constraints on the shape of the domain $\Om_t$ 
and on the velocity vector field $u$ (Theorem \ref{thm:Cauchy.pb}) 
and, in general, it is not preserved along the motion  
and can lead to illposed Cauchy problems (Theorem \ref{thm:not.invariant.subset}).
These constraints play a role also in the study of time-independent solutions. 

For time-independent $(\Om,u)$, problem \eqref{eq:system}-\eqref{eq:constant.curl} becomes 
\begin{equation} \label{steady.rotating.solution}
\begin{cases}
\la u, \nabla \ra  u=-\nabla p   & \text{in } \Omega,
\\
\div u=0 & \text{in } \Omega,
\\
\curl u= \alpha_0 e_3 & \text{in } \Omega,
\\
p=\sigma_0 H_{\pa \Omega} &\text{on } \pa \Omega,
\\
\la u ,\nu_{\pa\Om}\ra =0  &\text{on } \pa \Omega.
\end{cases}
\end{equation}
The main result of the paper is the following rigidity theorem.

\begin{theorem}\label{thm:global}
Let $\Omega$ be a strictly convex open set with $C^2$ boundary 
and $u\in C^2(\Omega,\R^3) \cap C(\overline{\Om}, \R^3)$. 
Suppose that $(\Omega,u)$ is a solution of \eqref{time.indep.sol}, \eqref{steady.rotating.solution}.
Then 
\begin{itemize}

\item[$(i)$]
There exists a real constant $\lambda_0$, a strictly convex open set $D\subset \R^2$ 
and a real function $f \in C^2(D) \cap C(\overline{D})$, 
with $f>0$ in $D$ and $f=0$ on $\pa D$, such that
\[
\Omega = \{(x', x_3) \in \R^3 : x' \in D, \ |x_3-\lambda_0| < f(x') \}.
\]
Moreover the vector field $u$ is independent of $x_3$, 
and it has third component $u_3 = 0$. 

\item[$(ii)$]
If 
\begin{equation}\label{eq:sigma.alpha}
\frac{\alpha_0^2}{\sigma_0} |D|^{\frac32} \leq \sqrt{2\pi^3}, 
\end{equation}
then $D$ is a disk, 
$\pa\Om$ is a surface of revolution, 
and the velocity vector field is $u = \frac{\a_0}{2} (-x_2, x_1, 0)$. 
\end{itemize}
\end{theorem}

\begin{remark}[\emph{Strict convexity is not really necessary}] 
As we highlight in Remark \ref{rmk:final} below, 
the assumption of strict convexity is actually not necessary to prove Theorem \ref{thm:global}, 
and the convexity of $\Omega$ would be sufficient. 
Nonetheless, we give the proof in details assuming $\Omega$ strictly convex
to simplify the proofs and the notations, 
and we  briefly outline the steps that lead to prove Theorem \ref{thm:global} when $\Omega$ is just convex. 
\end{remark}

\begin{remark}[\emph{Regularity assumption}] 
The regularity for $(\Om,u)$ required in Theorem \ref{thm:global} is only due to geometric arguments,
and it is weaker than that in Theorem \ref{thm:Cauchy.pb}, where $(\Om,u)$ is the initial datum 
of a Cauchy problem. 
\end{remark}

\begin{remark}[\emph{Formula of the surface $\pa \Om$}] 
\label{rem:f.Lopez}
In the second item of Theorem \ref{thm:global},  
the function $f$ giving the profile of the surface $\pa \Om$ 
can be determined from an explicit identity, see \eqref{eq:ODE}, 
leading to the axisymmetric shape of the capillary drop already obtained in \cite{Lo}. 
Concerning the connection between Theorem \ref{thm:global} and \cite{Lo}, 
see also the discussion below. 
\end{remark}

Theorem \ref{thm:global} is related to the results of \cite{Lo,W}, 
where the calculus of variations is used to study the shape of rotating fluids.  
To discuss this connection, a clarification about the use of the expression 
``stationary solution'' is in order: 

\begin{itemize}
\item 
\underline{in \cite{Lo,W}}, an open set $\Omega \subset \R^3$ of class $C^2$ is said to be 
a `stationary rotating solution' if it is a critical point of the shape functional
\begin{equation}\label{eq:Lo.W.}
\mathcal L(\Omega)= \sigma_0 P(\Omega) - \gamma_0 \int_{\Omega} |x'|^2\, dx,
\end{equation}
where $P(\Omega)$ is the perimeter of $\Omega$, 
among all sets of finite perimeter with fixed volume, with $\gamma_0 \in \R$;

\item
\underline{in the present paper}, 
a pair $(\Omega,u)$ is said to be a `steady solution', or `stationary solution', if 
$\Omega$ is an open bounded set with $C^2$ boundary, $u\in C^2(\Omega)$, 
and $(\Om,u)$ satisfies \eqref{time.indep.sol}, \eqref{steady.rotating.solution}.
\end{itemize}

%

In \eqref{eq:Lo.W.} we used the notation $x=(x',x_3)\in \R^2\times \R$.
We emphasize that, 
with respect to the full problems studied in the papers \cite{Lo, W}, 
in \eqref{eq:Lo.W.} we are making a big simplification: 
in \cite{W} the author also considers the possibility of self-gravitation, 
while in \cite{Lo} rotating stationary surfaces 
(i.e., not necessarily boundaries of smooth sets) are studied.
We note that the definition in \cite{Lo,W} 
only deals with the fluid domain $\Om$,   
as the velocity vector field $u$ does not appear in it, 
and that the adjective ``stationary'' in that definition 
reminds the expression ``stationary point of a functional''. 
On the other hand, the meaning of ``stationary'' 
used in the present paper is rather standard in fluid dynamics,   
and includes both the fluid domain $\Om$ and the velocity vector field $u$. 
We also recall that, in the fluid dynamics literature, the adjectives `steady' and `stationary' 
are often used synonymously, e.g., in the book \cite{AK} of Arnold and Khesin.

%

Now let us comment on the differences between these two definitions 
and the connection between the results. 
Computing the first variation of $\mathcal L(\Omega)$, 
one finds that a stationary rotating solution $\Omega$ in the sense of \cite{Lo,W}, 
that is, a critical point of the functional $\mL$, satisfies the Euler-Lagrange equation
\begin{equation}\label{eq:Lop}
\sigma_0 H_{\pa \Omega} - \gamma_0 |x'|^2 = \lambda  \quad \text{on } \pa \Omega,
\end{equation}
where $\lambda$ is a Lagrange multiplier that appears because of the volume constraint.
On the other hand, we start with the free boundary problem 
for the Euler equations of fluid dynamics with constant vorticity, 
and we prove (see identities \eqref{eq:stream}, \eqref{u.decomposition} 
and Lemma \ref{lem:craig.sulem} below) 
that, if $(\Om,u)$ solves \eqref{time.indep.sol}-\eqref{steady.rotating.solution}, then 
there exists a function $g$, independent of $x_3$, harmonic in $\Om$, 
and a constant $c$ such that 
\begin{equation}\label{eq:nostra}
\frac12 \Big| \nabla g - \frac{\alpha_0}{2} (x', 0) \Big|^2 
+ \alpha_0 \Big(g-\frac{\alpha_0}{4}|x'|^2\Big)  
+ \sigma_0 H_{\pa \Om} = c \quad  \text{on }  \pa \Omega.
\end{equation}

The main difference between \eqref{eq:Lop} and \eqref{eq:nostra} 
is the presence of the function $g$.  
In fact, our problem contains both geometry (the shape of the fluid domain $\Omega$) 
and dynamics (the velocity vector field $u$, which in \eqref{eq:nostra} is expressed in terms of $g$), 
while problem \eqref{eq:Lop} considers only the geometry part. 
If $g$ is a constant, then \eqref{eq:nostra} becomes \eqref{eq:Lop} 
with $\gamma_0 = \a_0^2/8$ and $\lm = c - \alpha_0 g$. 
In this paper we \emph{prove} that, when \eqref{eq:sigma.alpha} holds, 
$D$ is a disk, and from this we \emph{deduce} that $g$ is a constant. 
Thus, Theorem \ref{thm:global} states that if the equatorial planar region $D$ is convex 
and the ratio $\a_0^2 / \sigma_0$ satisfies \eqref{eq:sigma.alpha}, 
then the domain $\Om$ of the solution $(\Om, u)$ 
of \eqref{time.indep.sol}, \eqref{steady.rotating.solution}
and the critical point of the shape functional \eqref{eq:Lo.W.} coincide.

\medskip

\underline{{\it Sketch of the proof.}}
While it is clear that the problem of finding critical points $\Om$ of the shape functional \eqref{eq:Lo.W.} 
is a problem of calculus of variations,
it is less evident that studying solutions $(\Om,u)$ 
of \eqref{time.indep.sol}, \eqref{steady.rotating.solution}
is also a problem of calculus of variations. 
The first step towards the proof of Theorem \ref{thm:global} 
is to show that the set $\Omega$ must have a horizontal plane of symmetry, 
while $u$ must be independent of $x_3$ 
(see Theorem \ref{thm:Cauchy.pb}, now applied to time-independent solutions). 

At this point, since the domain $\Omega$ is assumed to be bounded and convex, 
it may seem natural to write the equations on a spherical reference manifold,
that is, in terms of functions defined on $\S^2$, like in \cite{BJL, BLL}.
However, this turns out to be not a convenient choice, 
because the symmetry properties proved in Theorem \ref{thm:Cauchy.pb} 
are not easy to handle when using nearly spherical sets. 
On the contrary, we find it more convenient to parametrize the boundary of $\Omega$ 
as the graph of a function $f$ (and its opposite $-f$) 
over the equatorial planar region $D$, see Lemma \ref{lem:craig.sulem}.

Observe that the symmetry of $\Omega$, joint with the assumption that $\Omega$ is smooth, 
immediately implies that the gradient of the profile function $f$ must blow up 
when approaching $\pa D$. 
This suggests that it is hard to use a strategy based on the implicit function theorem in Banach spaces. 
In fact, one might be tempted to prove Theorem \ref{thm:global} by showing that 
the trajectory $\alpha_0 \mapsto (\Omega(\alpha_0), u(\a_0))$, 
for values of the vorticity parameter $\alpha_0$ in a neighborhood of zero,
contains the only steady solutions close to $(\Om(0), u(0))$. 
In principle, this makes sense because we already observed that, 
for $\alpha_0=0$, the only solution $\Omega (0)$ is a ball, and $u(0) = 0$. 
However, since the gradient of the profile function $f$ blows up 
without an evident quantitative control,
it is hard to guess the correct function space where the implicit function theorem could be set.  

Parametrizing $\pa \Omega$ as the graph of $\pm f$ over $D$, we obtain 
a system of equations where all the unknown functions are defined 
on the unknown equatorial planar region $D$. 
We underline that this dimensional reduction is not the standard Craig-Sulem formulation 
of fluid dynamics as a problem on the boundary $\pa \Omega$.

The next step it to analyze the properties of the velocity field $u$, 
also using the time-independence assumption.
An accurate investigation of the problem shows that the torsion function $v_D$ 
of the equatorial planar region $D$, i.e., the solution of the problem 
\[ 
\Delta v_D = -1 \quad \text{in } D, \qquad 
v_D \in H^1_0(D),
\] 
plays an essential role. 
Indeed, we show that $D$ must support a solution to an overdetermined problem, 
namely a solution to a second order elliptic problem with both Dirichlet 
and Neumann-type conditions on $\pa D$, which is 
\begin{equation} \label{eq:over.determined.in.the.intro}
\frac{\alpha_0^2}{2}  | \nabla v_D|^{2}  
- \frac{ \sigma_0 \beta }{ | \nabla v_D| }  
+ \sigma_0 H_{\pa D}
= c 
\quad \ \text{on } \pa D,
\end{equation}
where $\beta,c$ are real constants, 
and $H_{\pa D}$ is the curvature of the boundary $\pa D$. 
It is not obvious that \eqref{eq:over.determined.in.the.intro} 
is strong enough to infer that $D$ is a disk, 
mainly because the constants $\beta$ and $c$ 
seem to be related to the nature of the original problem 
(which is three dimensional) and therefore to the profile function $f$. 
Nevertheless, an investigation of the smoothness assumption 
provides explicit formulas for $\beta$ and $c$ 
as functions of the planar set $D$, 
see Theorem \ref{thm:4.1} and Lemmas \ref{lem:constant.c}, \ref{lem:beta}.  
To conclude, we use special properties of the torsion function of a planar convex set 
to find an efficient bound for the constant $\beta$ (see \eqref{beta.1.2}).
This allows us to conclude that, if \eqref{eq:sigma.alpha} holds, 
then \eqref{eq:over.determined.in.the.intro} is sufficiently strong 
to guarantee that $D$ must be a disk. 
Once we have obtained this rigidity property, i.e., that $D$ is a disk, 
it is easy to check that the the profile function is radial and it can be 
computed by quadrature, see Remark \ref{rem:lopez}. 

\medskip

\underline{{\it Related literature}}. The axisymmetric equilibria configurations for capillary drops 
with constant vorticity were discovered experimentally by Plateau in \cite{Plateau} 
and computed by Poincaré \cite{Poincaré}, Rayleigh \cite{Rayleigh} 
and Chandrasekhar \cite{Chandrasekhar}, where their stability properties were also studied,
and more recently by Wente \cite{W} and Lopez \cite{Lo}. 

Concerning the free-boundary problems for fluids with constant vorticity, 
we mention for instance \cite{Wahlen, Wahlen.2007, Pasquali} 
for the Craig-Sulem formulation of the 2D ocean problem 
with or without general bottom topography, 
and \cite{LaS} for the 2D capillary drop problem with constant vorticity. 
The existence of global solutions for fluids with constant vorticity has been studied in 
\cite{Constantin.Strauss, Constantin.K.2009, Wahlen.2, Fan.Gao.2021}, 
and in \cite{BBMM, BFM} for periodic and quasi-periodic travelling waves. 
Travelling drops and bubbles in a two-fluid interaction have been constructed in \cite{MNS}, 
where an overdetermined problem also appears. 

Concerning rigidity results for the 3D fluid dynamics, 
we mention \cite{Wahlen} for the 3D ocean problem with constant vorticity, 
and \cite{Peralta.Salas} where localizable 3D Euler flows in bounded domains 
are proved to be axisymmetric, with rotationally symmetric domain 
whose transverse section is a disk or an annulus with convex boundary curves.


\medskip

\underline{{\it Structure of the paper}}.
In Section \ref{sec:prel} we report some definitions and preliminary results.
In Section \ref{sec:symmetry} we prove Theorems \ref{thm:Cauchy.pb}, \ref{thm:not.invariant.subset},
which also give a first reduction. 
In Section \ref{sec:necessary.condition} we show 
that if $D$ is convex, then the problem of finding a solution of \eqref{time.indep.sol}, 
\eqref{steady.rotating.solution} is connected to the overdetermined problem 
\eqref{eq:over.determined.in.the.intro} for the torsion function.
In Section \ref{sec:global} we conclude the proof of Theorem \ref{thm:global}. 

\medskip

\emph{Acknowledgements.} The authors thank Stefano Pasquali and Nicola Fusco 
for interesting discussions and nice suggestions.
This work is supported by GNAMPA 
and by University of Naples Federico II through FRA 2024 Geometric Topics in Fluid Dynamics.

\section{Notations and preliminary results} \label{sec:prel}
We start by briefly recalling  the notion of mean curvature, torsion function  and some useful formulas. 
For an exhaustive treatment of the subject, we refer, e.g., to \cite{Mantegazza}.

Let $d=2$ or $d=3$. For a bounded open set $E \subset \R^d$, 
we denote by $\nu_{\pa E}(x)$ the exterior unit normal. 
If $E$ is of class $C^2$, then there exists $\delta>0$ such that, 
defining $I_{\delta}(\pa E) := \{y\in \R^d: \, \mathrm{dist}(x, \pa E) < \delta\}$,
there exists $\tilde \nu \in C^1(I_\delta(\pa E))$  
that satisfies $\tilde \nu = \nu_{\pa E}$ on $\pa E$.
The mean curvature at $x \in \pa E$ is defined as the tangential divergence 
of the normal, namely 
\[
H_{\pa E} = \div_{\!\pa E} \nu_{\pa E}
= \div \tilde \nu - \la (D \tilde \nu) \,  \nu_{\pa E}, \nu_{\pa E} \ra,
\]
where 
$D \tilde \nu$ is the Jacobian matrix of the vector field $\tilde \nu$.
In dimension $d=2$, for a bounded convex open set $E \subset \R^2$, 
we recall the Gauss-Bonnet formula
\begin{equation} \label{Gauss.Bonnet}
\int_{\pa E} H_{\pa E}\,d\s = 2\pi.
\end{equation}
Observe that, with this definition, the mean curvature is positive if $E$ is convex, 
and $H_{\pa B} = d-1$, where $B$ is the unit ball of $\R^d$, 
and its boundary $\pa B = \S^{d-1}$ is the unit sphere.  
If $E$ is locally the subgraph of a $C^2(\R^{d-1})$ function $f$, at a point $(x',f(x'))$ it holds
\begin{equation}\label{eq:meancurvature.operator} 
H_{\pa E} (x',f(x'))= - \div \left( \frac{\nabla f(x')}{ \sqrt{1+|\nabla f(x')|^2}}\right)
=: H(f)(x').
\end{equation}
Moreover, if there exists a smooth function $g\in C^2_c(\R^d)$ such that 
$\Om = \{ x \in \R^d : g(x) > s \}$  
and $s$ is a regular value for $g$ (i.e., $\grad g(x) \neq 0$ for all $x$ such that $g(x) = s$),  
then for any $x\in \pa E$, it holds 
\begin{equation} \label{curvatura.level.set}
 H_{\pa E}(x) = -\div \Big( \frac{\nabla g(x)}{|\nabla g(x)|} \Big).
\end{equation}
We recall the tangential divergence theorem 
\begin{equation}\label{eq:tang.divergence}
\int_{\pa E} \div_{\pa E} X \, d\sigma
= \int_{\pa E} H_{\pa E} \la X, \nu_{\pa E} \ra \, d\sigma
\end{equation}
for any vector field $X\in C^1(\R^d)$.
Another useful integral identity involving the curvature is 
\begin{equation} \label{Reilly.in.sec.2}
\int_{\pa E} H_{\pa E} \la \grad u, \nu_{\pa E} \ra^2 \, d\sigma 
= \int_{E} \Big( (\Delta u)^2- |\nabla^2 u|^2 \Big) \, dx,
\end{equation}
where 
$u = 0$ on $\pa E$, 
$\grad^2 u = D^2 u$ is the Hessian matrix, and $|\nabla^2 u|^2$ is the sum of the square of its entries. 
This is a special case of Reilly formula \cite{Reilly.1977}; 
the special form \eqref{Reilly.in.sec.2} is, e.g., in \cite{Ros.1988}.

We recall an important symmetry result regarding the mean curvature of a set.
A first result in this direction is due to Li \cite{Lin} 
and it has been later generalized in the subsequent papers \cite{LN,LN1,LYY}).

\begin{theorem}[Li, Yan, Yao \cite{LYY}]\label{thm:Li}
Let $E \subset \R^d$ be a bounded open set of class $C^2$ 
such that, for any $x \in \pa E$ with $(\nu_{\pa E}(x))_d = 0$, 
there exists a vertical cylinder $C_r$ tangent to $\pa E$ at $x$ 
and such that $\overline C_r \cap \overline{E} \subset \pa E$. 
Assume further that for any two points $(x', a), \ (x',b) \in \pa E$ 
such that $\{(x', \lambda a + (1-\lambda)b ) : \lambda \in (0,1) \} \subset E$ it holds
\[
H_{\pa E}(x',a)= H_{\pa E}(x',b).
\]
Then there exists $\lambda\in \R$ such that $E$ is symmetric 
with respect to the plane $\{x_{d}=\lambda\}$.
\end{theorem}

Actually, in \cite{LYY} Theorem \ref{thm:Li} is stated in a more general context, 
assuming that the mean curvature is nonincreasing in the $x_d$ direction; 
here, instead, we state the simplified version that is used in the present paper. 

We also recall the notion of torsion function, 
which is an essential tool for our analysis. 
Let $E \subset \R^d$  be an  open bounded  set.  
We denote by $v_E$ the torsion function of the set $E$, 
which is the unique solution of 
\[
\begin{cases}
\Delta v_E = -1 \ \ \text{ in } E, 
\\
v_{E} \in H^1_0(E).
\end{cases}
\] 
Some properties of this function are immediate to obtain.
In fact, by maximum principle, it is immediate to show that 
$v_E(x) \geq 0$ and equality holds if and only if $x\in \pa E$. 
Moreover, since $v_E$ is constant on $\pa E$, for any $x \in \pa E$ it holds 
\begin{equation}\label{torsion.gradient}
\nabla v_E(x) = -|\nabla v_E(x)| \nu_{\pa E}(x).
\end{equation}
Other useful identities about the torsion function are
\begin{align}
& \int_E v_E \, dx = \int_E v_E (-\Delta v_E) \, dx 
= \int_E |\grad v_E|^2 \, dx, 
\label{torsion.integral} 
\\ 
& \int_{\pa E} |\nabla v_E| \, d\sigma 
= \int_{\pa E} - \la \grad v_E, \nu_{\pa E} \ra \, d\sigma 
= - \int_E \Delta v_E \, dx 
= |E|. 
\label{torsion.gradient.integral}
\end{align}

A rigidity result concerning the torsion function is the following theorem by Serrin.

\begin{theorem}[Serrin \cite{S}, Theorem 3] 
\label{thm:serrin}
Let $E$ be a bounded open set in $\R^d$ with $C^3$ boundary. 
Let $v_E$ be its torsion function. 
If there exists a nonincreasing function $\mathrm f$ such that
\begin{equation} \label{serrin.f.H}
\frac{\pa v_E}{\pa \nu_{\pa E}} = \mathrm f(H_{\pa E})\quad \text{ on }\pa E,
\end{equation}
then $E$ is a ball.
If $|\nabla v_E|$ is a nondecreasing function of $H_{\pa E}$ on $\pa E$, 
then $E$ is a ball. 
\end{theorem}

When $E$ is convex, the torsion function has some very special properties. 
To this regard, we state the following theorems. 

\begin{theorem}[Makar-Limanov \cite{ML}, Kawohl \cite{Kaw}, Kennington \cite{Ken}]
\label{thm:torsion.square.root.concave}
Let $E$ be a bounded convex open set with $C^2$ boundary. 
Then $v_E$ is $1/2$-concave, meaning that $\sqrt{v_E}$ is concave.
\end{theorem}

Theorem \ref{thm:torsion.square.root.concave}  has been  first  proved in \cite{ML} for $d=2$, 
then in \cite{Kaw} and \cite{Ken} for $d>2$. 
Finally, we also recall the gradient estimate for the torsion function of planar convex open sets.

\begin{theorem}[Classical; see, e.g., \cite{BBBH, HS}]
\label{teo:grad.estimate}
Let $E \subset \R^2$ be an open convex set. 
Then the torsion function $v_E$ satisfies
\[
\| \nabla v_E \|_{L^\infty(E)} \leq c |E|^\frac12
\]
for some universal constant $c < (2\pi)^{-\frac{1}2}$.
\end{theorem}

The proof of Theorem \ref{teo:grad.estimate} is quite classical;  
many works in the recent years have been devoted to find the optimal value   
of the constant $c$, since it is related to the Hermite-Hadamard inequality.  
We mention, for instance, the papers \cite{BBBH} and \cite{HS}.

\section{Time evolution and constant vorticity}
\label{sec:symmetry}

In this section we prove our results about the compatibility of the 
constant vorticity condition \eqref{eq:constant.curl} 
with the equations \eqref{eq:system} for the time evolution of the capillary drop, 
namely we prove Theorems \ref{thm:Cauchy.pb} and \ref{thm:not.invariant.subset}.

\begin{proof}[\textbf{Proof of Theorem \ref{thm:Cauchy.pb}.}]
The existence of a $\delta>0$ such that a solution to \eqref{eq:system} exists in $(-\delta,\delta)$
is guaranteed by \cite[Theorem 1.1]{CS}.
Assume that
the velocity field $u$ has constant vorticity $ w=\curl u= \alpha_0 e_3$ for $t\in (-\delta, \delta)$. This implies that 
\begin{equation}\label{eq:curl.u}
\begin{pmatrix}
\pa_2u_3-\pa_3u_2 
\\
\pa_3 u_1 - \pa_1 u_3
\\
\pa_1 u_2 - \pa_2u_1
\end{pmatrix}
= \alpha_0 
\begin{pmatrix}
0\\ 0\\ 1
\end{pmatrix}.
\end{equation}
Moreover, the vector field $w= \curl u$ satisfies the vorticity equation 
\[
\pa_t w + \la u , \grad \ra w  
= \langle w, \nabla \rangle u.
\]
Since $w = \alpha_0 e_3$ is constant (independent of $t$ and $x$), we have
\begin{equation} \label{eq:vort.equation}
0=\langle w, \nabla \rangle u =\alpha_0
\begin{pmatrix}
\pa_3 u_1 \\
\pa_3 u_2 \\
\pa_3 u_3
\end{pmatrix}.
\end{equation}
Since $\alpha_0\not =0$,
the velocity field $u$ does not depend on the third variable, i.e., $u(t,x_1,x_2,x_3)=u(t,x_1,x_2)$. 
Thus, \eqref{eq:curl.u} implies that $\pa_2 u_3$ and $\pa_1 u_3$ vanish, 
so that $\grad u_3 = 0$, that is, 
$u_3(t,x) = u_3(t)$ is independent of $x$.  
This implies that $u_3=0$, because $\int_{\Omega_t}u=0$. 
Since $u_3=0$, from the first equation in \eqref{eq:system} we deduce that
\begin{equation} \label{eq:pressure.equation}
\pa_3 p=0,
\end{equation} 
whence $p(t,x_1,x_2,x_3)=p(t,x_1,x_2)$.  
Moreover, the third equation in \eqref{eq:system} implies the following symmetry property:
\begin{equation}\label{eq:curvature.symmetric}
\text{if } (x_1,x_2,x_3) \text{ and } (x_1,x_2, y_3) \in \pa \Omega_t 
\ \ \Longrightarrow \ \ 
H_{\pa \Omega_t}(x_1,x_2,x_3)= H_{\pa \Omega_t} (x_1,x_2,y_3).
\end{equation}
Since $\Omega_0$ is convex, the assumption of Theorem \ref{thm:Li} are satisfied 
and hence we find that there exists $\lambda_0 \in \R$ such that 
the set $\Omega_0$
is symmetric with respect to the plane $\{x_3=\lambda_0\}$.
Since $\Omega_0$ is symmetric and strictly convex, 
there exists a concave function 
\[
f \in C^5(D,\R) \cap C(\overline{D},\R), \quad \ 
f > 0 \text{ in } D, \quad \ 
f=0 \text{ on } \pa D, 
\] 
where $D$ is the planar set $\Om_0 \cap \{ x_3 = \lm_0 \}$, or, more precisely, 
\[
D := \{ x' \in \R^2 : (x',\lm_0) \in \Omega_0 \},
\] 
such that
\[
\pa \Omega_0 
= \{(x', \lambda_0 +f(x')) : x' \in D\} 
\cup  \{(x', \lambda_0 -f(x')) : x' \in D\}
\cup \Sigma_0,
\]
where
\[
\Sigma_0= \pa D \times\{\lambda_0\}.
\]
In fact, from the convexity of $\Om_0$ it is immediate to show that 
$\pa \Omega_0 \cap \{x_3>\lambda_0\}$ is a graph over $D$.  
The regularity of the function $f$ comes from the regularity of the boundary $\pa \Omega$.
\end{proof}

\begin{proof}[\textbf{Proof of Theorem \ref{thm:not.invariant.subset}.}]
Let $\Omega_0 = B_1$, $u_0 = \frac{\alpha_0}{2} (-x_2, x_1, 0)$, 
and let $(\Om_t, u)$ be the unique smooth solution of the Cauchy problem for \eqref{eq:system} 
in some time interval $(-\delta, \delta)$, $\delta > 0$, 
with initial data $(\Omega_0, u_0)$ at time $t=0$. 
We calculate 
\begin{equation*}
\div (u_0\cdot\grad u_0) 
= - \frac{\a_0^2}{4} \, \div \!\! \begin{pmatrix}x_1 \\ x_2 \\ 0 \end{pmatrix}
= - \frac{\a_0^2}{2}.
\end{equation*}
Hence, taking the divergence of the first equation in \eqref{eq:system}, 
and recalling that the curvature of the unit sphere $\S^2$ is 2,   
we get that the pressure at time $t=0$ verifies 
\[
\begin{cases}
\Delta p = \a_0^2/2  
& \text{in } \Om_0 = B_1,
\\
p = 2 \s_0  
& \text{on } \pa \Om_0 = \S^2. 
\end{cases}
\]
The solution to this problem is 
\begin{equation*}
p(0,x) = \frac{\a_0^2}{12}(x_1^2 + x_2^2 + x_3^2 - 1) + 2\s_0,
\end{equation*}
which depends nontrivially on $x_3$, 
with $\pa_3 p(0,x) = (\a_0^2 / 6) x_3$.  
If $\curl u$ is constant in the time interval $(-\delta, \delta)$, 
then, as proved above, \eqref{eq:pressure.equation} holds, and this is a contradiction. 
\end{proof}

\begin{remark}[Regularity and convexity assumptions]
\label{rem:convex.hyp}
As already explained in the introduction, 
the study of time-independent solutions requires much less regularity. 
In fact, since there is no need of using the short time existence result of Coutand and Shkoller, 
the assumption $\Omega$ of class $C^2$ and strictly convex 
is sufficient to guarantee that $\Omega_0$ admits a plane of symmetry, 
which is enough to make the argument of the next sections work.

We also observe that the strict convexity of $\Omega_0$ is actually not needed, 
but it makes the presentation less heavy. This is discussed in Remark \ref{rmk:final} below.  
\end{remark}

\section{An overdetermined problem and rigidity}
\label{sec:necessary.condition}

In this section we begin the study of the time-independent problem 
\eqref{steady.rotating.solution}, starting with the following basic observation. 

\begin{lemma} \label{lemma:proof.thm.1.7.i}
Assume the hypotheses of Theorem \ref{thm:global}. 
Then item $(i)$ of Theorem \ref{thm:global} holds. 
\end{lemma}

\begin{proof} 
Repeat the proof of Theorem \ref{thm:Cauchy.pb}. 
The only difference is that now $\Om$ and $u$ are of class $C^2$ 
and independent of time. 
\end{proof}

By item $(i)$ of Theorem \ref{thm:global}, the domain $\Om$ is symmetric with respect to 
the plane $\{ x_3 = \lm_0 \}$.  
Without loss of generality,  we assume that the barycenter of $\Omega$ is at the origin, 
and therefore $\lm_0 = 0$, that is, $\Omega$ is symmetric with respect to the plane $\{x_3=0\}$. 

The starting point of our analysis is following theorem about the torsion function. 

\begin{theorem}[An overdetermined problem for the torsion function]
\label{thm:4.1}
Let $(\Omega,u)$ satisfy the assumptions of Theorem \ref{thm:global}, 
and suppose that the barycenter of $\Omega$ is at the origin. 
Let $\lm_0, D, f$ be given by item $(i)$ of Theorem \ref{thm:global}, 
so that $\lm_0 = 0$, 
\begin{equation} \label{eq:omega.is.graph}
D = \{ x' \in \R^2 : (x',0) \in \Omega \}, 
\quad \ 
\pa \Omega \cap \{x_3>0\} = \{(x', f(x')) : x' \in D\}.
\end{equation}
Then the torsion function $v_D \in H^1_0(D)$ of the planar set $D$, 
that is, the solution of the problem 
\begin{equation} \label{eq:torsion} 
\begin{cases}
\Delta v_D = -1 & \text{in } D, \\
v_D=0 & \text{on } \pa D,
\end{cases}
\end{equation}
also satisfies the equation 
\begin{equation} \label{eq:over.determined}
\frac{\alpha_0^2}{2}  | \nabla v_D|^{2}  
- \frac{ \sigma_0 \beta }{ | \nabla v_D| }  
+ \sigma_0 H_{\pa D}
= c 
\quad \ \text{on } \pa D,
\end{equation}
where $H_{\pa D}$ is the curvature of the boundary $\pa D$ (which is a planar curve),
and $\beta$ and $c$ are real constants depending only on $D$.
\end{theorem}

The proof of Theorem \ref{thm:4.1} is divided in two parts: 
in Subsection \ref{sub:equation} we prove that 
$v_D$ satisfies \eqref{eq:over.determined}, 
while in Subsection \ref{sub:constant} we compute the values 
of the constants $\beta$ and $c$ and show that they  
depend only on $D$.

\subsection{A geometric-analytic condition}\label{sub:equation}
The main goal of this subsection is to prove the first part of Theorem \ref{thm:4.1}. 
Let $\Omega, u, D, f$ be like in Theorem \ref{thm:4.1}.
We introduce the matrices
\begin{equation} \label{def.mJ.2.3}
\mJ_3 :=
\begin{pmatrix}
0 &-1& 0
\\
1&0&0
\\
0&0& 0
\end{pmatrix}, 
\quad \ 
\mJ_2 :=
\begin{pmatrix}
0 &-1 
\\
1&0
\end{pmatrix}.
\end{equation}
We decompose the vector field $u$ into its irrotational and rotation part, 
\begin{equation}\label{eq:ansatz}
u = v + \frac{\a_0}{2} \mJ_3 x, 
\quad \mJ_3 x = \begin{pmatrix} - x_2 \\ x_1 \\ 0 \end{pmatrix},
\end{equation}
where $v$ is defined as the difference $v := u - (\alpha_0/2) \mJ_3 x$.
Since $u$ is independent of $x_3$ and it has third component $u_3 = 0$,
also $v$ is independent of $x_3$ and $v_3 = 0$.  
Moreover, $\div(\mJ_3 x) = 0$ and $\curl(\mJ_3 x) = 2 e_3$, 
so that both $\div v$ and $\curl v$ vanish in $\Om$.
We define the planar vector fields 
\begin{gather}
\tilde u : D \to \R^2, \quad \ 
\tilde v : D \to \R^2, 
\notag \\
\tilde u(x') 
= \begin{pmatrix} \tilde u_1(x') \\  \tilde u_2(x') \end{pmatrix} 
= \begin{pmatrix} u_1(x',0) \\ u_2(x',0) \end{pmatrix}, 
\quad \ 
\tilde v(x') 
= \begin{pmatrix} \tilde v_1(x') \\  \tilde v_2(x') \end{pmatrix} 
= \begin{pmatrix} v_1(x',0) \\ v_2(x',0) \end{pmatrix}, 
\label{def.tilde.u.tilde.v}
\end{gather}
and we note that 
\[
\div_2 \tilde u = \div_2 \tilde v = 0, \quad \ 
\curl_2 \tilde u = \alpha_0, \quad \ 
\curl_2 \tilde v = 0 \quad \ 
\text{in } D,
\]
where $\div_2$ and $\curl_2$ are the $\R^2$-divergence and $\R^2$-curl, 
i.e.,
\[
\div_2 F := \pa_{x_1} F_1 + \pa_{x_2} F_2, \quad \  
\curl_2 F := \pa_{x_1} F_2 - \pa_{x_2} F_1  
\]
for any planar vector field $F \in C^1(D, \R^2)$. 
Since $D \subset \R^2$ is convex, and therefore simply connected,  
and $\curl_2 \tilde v = 0$ in $D$, 
there exists a scalar function $\phi: D \to \R$ such that 
\begin{equation} \label{def.phi}
\tilde v = \grad \phi 
\quad \ \text{in } D.
\end{equation}
The function $\phi$ is harmonic in $D$, 
because $\Delta \phi = \div \grad \phi = \div \tilde v = 0$ in $D$,  
and therefore there exists a function $g : D \to \R$ such that 
\begin{equation}\label{eq:stream}
\nabla g = \mJ_2 \nabla \phi \quad \ \text{in } D. 
\end{equation}
Both $\phi$ and $g$ are functions of two variables $(x_1, x_2) = x' \in D$, 
and they can be trivially extended to functions defined in $D \times \R$, 
and hence in $\Omega$. 
We also denote $\phi, g$ the extended functions, 
\[
\phi : \Omega \to \R, \quad 
\phi(x) = \phi(x', x_3) := \phi(x'), \quad 
g : \Omega \to \R, \quad  
g(x) = g(x', x_3) := g(x'). 
\]
Of course $\pa_{x_3} \phi = \pa_{x_3} g = 0$. 
Thus 
\[
v 
= \begin{pmatrix} v_1 \\ v_2 \\ 0 \end{pmatrix} 
= \grad \phi 
= \begin{pmatrix} \pa_{x_1} \phi \\ \pa_{x_2} \phi \\ 0 \end{pmatrix} 
\quad \text{in } \Omega, 
\]
and the decomposition \eqref{eq:ansatz} becomes
\begin{equation}\label{u.decomposition}
u = \grad \phi + \frac{\a_0}{2} \mJ_3 x 
\quad \text{in } \Omega,
\qquad  
\tilde u = \grad \phi + \frac{\alpha_0}{2} \mJ_2 x' 
\quad \text{in } D.
\end{equation}

The first step towards the proof of Theorem \ref{thm:4.1} 
is to deduce a system of equations for $f$ and $\phi$ on the planar set $D$. 
This is the content of the next lemma.

\begin{lemma} \label{lem:craig.sulem} 
Let $\Omega, u, D, f$ be like in Theorem \ref{thm:4.1}, 
and let $\tilde u, \tilde v, \phi,g$ be the functions defined in 
\eqref{def.tilde.u.tilde.v}, \eqref{def.phi}, \eqref{eq:stream}.  
Then
\begin{align} 
\frac12 |\tilde u|^2 + \alpha_0 g - \frac{\alpha_0^2}{4} |x'|^2 + \sigma_0  H(f) = c_1 
& \quad \ \text{in } D,
\label{eq:first.reduction.i}
\\
\langle \tilde u, \nabla f \rangle  
= \langle \nabla \phi  +\frac{\alpha_0}{2}\mJ_2 x', \nabla f\rangle=0  
& \quad \ \text{in } D,
\label{eq:first.reduction.ii}
\\
\la \tilde u, \nu_{\pa D} \ra 
= \la \nabla \phi +\frac{\alpha_0}{2} \mJ_2 x', \nu_{\pa D} \ra=0 
& \quad \ \text{on } \pa D,
\label{eq:first.reduction.iii}
\end{align}
for some constant $c_1\in \R$, 
where $H(f)$ is the mean curvature operator 
defined in \eqref{eq:meancurvature.operator}. 
\end{lemma}

\begin{proof} 
For any planar vector field $F \in C^1(D,\R^2)$, one has
\[
\langle F,\nabla\rangle F - \frac{\nabla (|F|^2)}{2}
= \begin{pmatrix}
F_2 \pa_2 F_1 - F_2 \pa_1 F_2 \\
F_1 \pa_1 F_2 - F_1 \pa_2 F_1
\end{pmatrix}
= (\pa_1 F_2 - \pa_2 F_1)
\begin{pmatrix}
- F_2 \\
F_1
\end{pmatrix}
= (\curl_2 F) \mJ_2 F.
\]
Applying this identity to $\tilde u$, 
recalling that $\curl_2 \tilde u = \alpha_0$, 
and using \eqref{u.decomposition}, \eqref{eq:stream},
we get 
\[
\langle \tilde u,\nabla \rangle \tilde u 
- \frac{ \nabla (|\tilde u|^2) }{2} 
= \alpha_0  \mJ_2 \tilde u
= \alpha_0  \mJ_2 \Big(\nabla \phi + \frac{\alpha_0}{2} \mJ_2 x' \Big)
= \alpha_0 \nabla g - \frac{\alpha_0^2}{2} x' 
= \grad \Big( \alpha_0 g - \frac{\alpha_0^2}{4} |x'|^2 \Big).
\] 
Hence the first equation in \eqref{steady.rotating.solution} 
on $\Om \cap \{ x_3 = 0 \} = D \times \{ 0 \}$ gives 
\[
\grad \Big( \frac{ |\tilde u|^2 }{2} + \alpha_0 g - \frac{\alpha_0^2}{4} |x'|^2 + \tilde p \Big) = 0 
\quad \text{in } D,
\]
where $\tilde p : D \to \R$ is the function $\tilde p(x') := p(x',0)$. 
As a consequence, there exists a constant $c_1 \in \R$ such that
\[
\frac{|\tilde u|^2}{2} + \alpha_0 g - \frac{\alpha_0^2}{4} |x'|^2 + \tilde p = c_1
\quad \text{in } D.
\]
Since $u_3=0$, from the first equation in \eqref{steady.rotating.solution} 
it follows that the pressure $p$ is independent of $x_3$. 
This means that the pressure is constant along any vertical line, 
and, in particular, for any $x' \in D$, 
the value of the pressure at the point $(x',0) \in \Om \cap \{ x_3 = 0\}$ 
is equal to its value at the point $(x', f(x')) \in \pa \Om \cap \{ x_3 > 0 \}$. 
Thus, by the fourth equation in \eqref{steady.rotating.solution}, 
\[
\tilde p(x') = p(x',0) = p(x', f(x')) = \sigma_0 H_{\pa \Omega}(x', f(x')) 
= \sigma_0 H(f)(x'),
\]
where $H(f)$ is defined in \eqref{eq:meancurvature.operator}. 
This proves \eqref{eq:first.reduction.i}. 

The fifth equation in \eqref{steady.rotating.solution} 
is the orthogonality property $\la u, \nu_{\pa\Om}\ra=0$ on $\pa \Omega$. 
The exterior normal vector to the boundary $\pa \Omega$ 
at the point $(x',f(x')) \in \pa \Om$, with $x'\in D$, is 
\begin{equation} \label{normal.pa.Om.f}
\nu_{\pa \Om} (x',f(x'))=
\frac{1}{\sqrt{1+|\nabla f(x')|^2}} \begin{pmatrix} - \nabla f(x') \\ 1 \end{pmatrix}.
\end{equation}
Since $u$ is independent of $x_3$ and its third component is $u_3 = 0$, 
recalling the definition \eqref{def.tilde.u.tilde.v} of $\tilde u$, one has 
\[
u(x', f(x')) = u(x',0) = \begin{pmatrix} \tilde u(x') \\ 0 \end{pmatrix}.
\]
Hence the identity $\la u, \nu_{\pa \Omega} \ra = 0$ on $\pa \Omega$ 
implies that $\la \tilde u, \grad f \ra = 0$ in $D$. 
Moreover $\tilde u$ is given by the second identity in \eqref{u.decomposition}, 
and this gives \eqref{eq:first.reduction.ii}.


Identity \eqref{eq:first.reduction.iii} is obtained from the fifth identity 
of \eqref{steady.rotating.solution} on $\pa \Omega \cap \{x_3=0\} = \pa D \times \{0\}$, 
observing that, since $\Omega$ is symmetric and smooth, one has
\begin{equation} \label{normal.at.pa.D}
\nu_{\pa \Omega}(x',0)=
\begin{pmatrix} 
\nu_{\pa D}(x')
\\
0
\end{pmatrix} \,\,\, \text{for } (x',0) \in \pa \Omega \cap \{x_3=0\}=\pa D \times \{0\},
\end{equation}
and using \eqref{u.decomposition}, as above, to replace $\tilde u$. 
\end{proof}

From now on, if there is no risk of ambiguity, we write $x$ instead of $x'$ when referring to points in $D\subset\R^2$.

\begin{remark}[\emph{The axisymmetric case}] 
\label{rem:lopez}
If $D$ is the disk $B_\rho$ of center $0$ and radius $\rho>0$, 
system \eqref{eq:first.reduction.i}, \eqref{eq:first.reduction.ii}, \eqref{eq:first.reduction.iii}
becomes easy to solve. Indeed, \eqref{eq:first.reduction.iii} becomes
\[
\la \nabla \phi , \nu_{\pa B}\ra = 0 
\quad \text{on } \pa B_\rho,
\]
because $\nu_{\pa B_\rho}(x) = x/\rho$ 
and $\la \mJ_2 x, x \ra = 0$. 
Hence $\phi$ is harmonic in the disk $B_\rho$ with zero Neumann boundary data, 
and therefore $\phi$ is constant. 
Thus $\grad \phi = 0$ in $B_\rho$, 
and $\grad g = \mJ_2 \grad \phi = 0$, 
whence $g$ is also a constant, $g = g_0 \in \R$. 
Since $\grad \phi = 0$, from \eqref{u.decomposition} we get 
$\tilde u = (\alpha_0/2) \mJ_2 x$ in $B_\rho$, 
whence $|\tilde u|^2 = (\alpha_0^2 / 4) |x|^2$, 
and \eqref{eq:first.reduction.i} becomes 
\begin{equation}\label{eq:exact0}
-\frac{\alpha_0^2}{8}|x|^2 +\sigma_0 H(f)=c_1-\alpha_0 g_0 
\quad \text{in } B_\rho. 
\end{equation}
Since $\grad \phi = 0$, 
\eqref{eq:first.reduction.ii} gives $\la \nabla f, \mJ_2 x \ra = 0$ in $B_\rho$, 
which implies that $f$ is a radial function. 

Since $f$ is radial and $D = B_\rho$ is a disk, 
equation \eqref{eq:exact0} can be easily transformed into an \textsc{ode}, 
which is the one studied by Lopez in \cite{Lo}, Section 4. 
For sake of completeness, we now compute the exact solution in this case.
We denote $r=|x|$.
Since $f$ is radial, we have 
$f(x) = f_0(r)$, where $f_0(s) := f(s,0)$, and 
\[
\nabla f(x) = f_0'(r) \frac{x}{r}, \quad \ 
\frac{\nabla f(x)}{\sqrt{1+|\nabla f(x)|^2}}
= \psi(r) x, \quad \ 
\psi(r) := \frac{f_0'(r)}{r \sqrt{1+f_0'(r)^2}}.
\]
Since $\div (\psi(r) x) = r \psi'(r) + 2 \psi(r)$, 
recalling the definition \eqref{eq:meancurvature.operator} of $H(f)$, 
\eqref{eq:exact0} becomes
\[
r \psi'(r) + 2 \psi(r) + \frac{\alpha_0^2}{8 \sigma_0} r^2 + \frac{c_1-\alpha_0 g_0}{\sigma_0} = 0.
\]
This is a linear \textsc{ode}. Its general solution in the interval $r \in (0,\infty)$ is 
\[
\psi(r) = \frac{\gamma_1}{r^2} - \frac{\alpha_0^2}{32 \sigma_0} r^2 - \frac {c_1 -\alpha_0 g_0}{2 \sigma_0}, 
\quad \ \gamma_1 \in \R.
\]
Recalling the definition of $\psi(r)$, and multiplying by $r$, we have
\begin{equation}\label{eq:ODE}
\frac{f_0'(r)}{\sqrt{1+f_0'(r)^2}} 
= r \psi(r) 
= \frac{\gamma_1}{r} - \frac{\alpha_0^2}{32 \sigma_0} r^3 - \frac{c_1 -\alpha_0 g_0}{2 \sigma_0} r.
\end{equation}
Since the quantity on the left is bounded, we find that $\gamma_1=0$. 
This implies that $f_0'(0)=0$.
Since the tangent plane to $\pa \Omega$ is vertical at points on 
$\pa \Omega \cap \{ x_3 = 0 \} = \pa B_\rho \times \{ 0 \}$, 
we have to impose that $\lim_{r \to \rho^-} f_0'(r) = - \infty$. 
This holds if the quantity on the right in \eqref{eq:ODE} is $-1$ at $r=\rho$, i.e., 
\[
\frac{ c_1 - \alpha_0 g_0 }{ 2\sigma_0 } 
= \frac{1}{\rho} - \frac{\a_0^2 \rho^2}{32 \sigma_0}.
\]
Thus, 
\begin{equation}\label{eq:ODE.bis}
\frac{f_0'(r)}{\sqrt{1+f_0'(r)^2}} 
= - \frac{\alpha_0^2}{32 \sigma_0} r^3 - \Big( \frac{1}{\rho} - \frac{\a_0^2 \rho^2}{32 \sigma_0} \Big) r.
\end{equation}
One has to invert the formula for $f_0'(r)$, 
and the solution $f_0(r)$ can be found by quadrature 
after imposing the boundary condition $f_0(\rho)=0$. 
\end{remark}

\begin{remark}[\emph{A threshold for the shape of monotone axisymmetrix profiles}] 
From formula \eqref{eq:ODE.bis} we note that $f_0'(r)$ is negative for all $r \in (0,\rho)$ 
if the coefficient of $r$ is $\leq 0$, that is, if 
\[
\frac{\a_0^2 \rho^3}{32 \sigma_0} \leq 1.
\]
This is a threshold for the shape of an axisymmetric rotating capillary drop: 
when the radius $\rho$ becomes too large, 
or the vorticity $\alpha_0$ too strong, 
or the capillarity $\sigma_0$ too weak, 
then the profile function $f_0(r)$ is no longer a strictly monotone decreasing function 
of the distance $r$ from the axis of symmetry. This means that the fluid domain $\Omega$ 
ceases to be a convex set. This is a well-known phenomenon in physics literature. 
\end{remark}

After the last remarks, now we go back to the general case.
When $D$ is not a disk, we also have to deal with the irrotational part $\nabla \phi$ of the velocity field. 
This makes the computations heavier.  
However, when the set $D$ is convex, the level sets of the torsion function of $D$ 
play a role analogous to that of polar coordinates in Remark \ref{rem:lopez}, 
as we show is the following Lemma. 

\begin{lemma} \label{lemma:tilde.f}
Let $\Omega, u, D, f$ be like in Theorem \ref{thm:4.1}, 
and let $v_D$ be the torsion function of $D$. Then there exist
a constant $c \in \R$ and a function
\[
\tilde f : [0,M] \to [0, \infty), \quad \  
M := \max_D v_D = \| v_D \|_{L^\infty(D)},
\] 
with $\tilde f \in C^{2}((0,M)) \cap C([0,M])$, 
$\tilde f(0)=0$,
such that $f = \tilde f \circ v_D$ in $\overline D$, and
\begin{equation}\label{eq:first.reduction1}
\frac{\alpha_0^2}{2} |\nabla v_D|^2 
+ \alpha_0^2 v_D 
- \sigma_0  \div \bigg( \frac{\tilde f'(v_D) \nabla v_D }{\sqrt{ 1 + | \tilde f' (v_D) \nabla v_D|^2}} \bigg) 
= c
 \quad\text{in } \, D.
\end{equation}
\end{lemma}

\begin{proof}
In $D$, let 
\begin{equation}\label{g1.def}
g_1 := g - \frac{\alpha_0}{4} |x|^2.
\end{equation}
Recalling \eqref{eq:stream} and \eqref{u.decomposition}, 
\begin{equation} \label{grad.g1.J2.tilde.u}
\nabla g_1 
= \nabla g -\frac{\alpha_0}{2} x
= \mJ_2 \big( \nabla \phi +\frac{\alpha_0}{2} \mJ_2 x \big)
= \mJ_2 \tilde u 
\quad \ \text{in } D.
\end{equation}
By the orthogonality property \eqref{eq:first.reduction.iii}, 
\[
\la \nabla g_1, \mJ_2 \nu_{\pa D} \ra 
= \la \mJ_2 \tilde u, \mJ_2 \nu_{\pa D} \ra
= \la \tilde u , \nu_{\pa D} \ra 
= 0 
\quad \ \text{on } \pa D.
\]
Since $\mJ_2 \nu_{\pa D}$ is tangent to $\pa D$, the last orthogonality property 
implies that $g_1$ is constant along the boundary $\pa D$, 
that is, there exists a constant $k \in \R$ such that $g_1=k$ on $\pa D$.
Thus, since $g$ is harmonic in $D$, the function $g_1$ solves the Dirichlet problem
\[
\begin{cases}
\Delta g_1 = - \alpha_0 & \text{in } D, \\
g_1 = k & \text{on } \pa D.
\end{cases}
\]
Hence $g_1$ is a translation of a multiple of the torsion function of $D$. 
In fact, the function $g_2 := (g_1 - k)/\alpha_0$ satisfies 
$\Delta g_2 = -1$ in $D$ and $g_2 = 0$ on $\pa D$, 
namely $g_2$ solves \eqref{eq:torsion}, 
and therefore it is the torsion function $v_D$ of the set $D$. 
Thus,  
\begin{equation} \label{g.1.v.D}
v_D = \frac{g_1-k }{\alpha_0}, \quad \ 
g_1 = k + \alpha_0 v_D.
\end{equation}
Since $D$ is convex, by Theorem \ref{thm:torsion.square.root.concave}, $\sqrt {v_D}$ is a concave function. 
Moreover, the superlevel sets $\{v_D\geq s\}$ are convex for any $s\in [0, M]$, 
$v_D$ has a unique critical point $x_* \in D$, which is the maximum point 
(i.e., $v_D(x_*) = M$), and the level sets of $v_D$ are smooth closed curves.  
Hence the vectors $\grad g_1 = \alpha_0 \grad v_D$ and $\mJ_2 \grad g_1 = \alpha_0 \mJ_2 \grad v_D$, 
except at the point $x_*$, are nonzero orthogonal vectors.
For any $s \in [0,M)$, consider the level set
\begin{equation} \label{def.mE.s}
\mE_s := \{ x \in D : v_D(x) = s \}
= \{ x \in D : g_1(x) = k + \alpha_0 s \}.
\end{equation}
For any point $x \in \mE_s$, 
the vector $\grad g_1(x)$ is normal to $\mE_s$ at $x$, 
and therefore $\mJ_2 \grad g_1(x)$ is tangent to $\mE_s$ at $x$. 
By \eqref{grad.g1.J2.tilde.u}, $\mJ_2 \grad g_1 = \mJ_2^2 \tilde u = - \tilde u$, 
and, by \eqref{eq:first.reduction.ii},
\[
\la \grad f , \mJ_2 \grad g_1 \ra 
= - \la \grad f , \tilde u \ra 
= 0 
\quad \ \text{in } D.
\]
Since $\mJ_2 \grad g_1$ is tangent to $\mE_s$, 
the last orthogonality property implies that $f$ is constant along the curve $\mE_s$. 
Denote $\tilde f(s)$ the value of $f$ on $\mE_s$, that is, $f(x) = \tilde f(s)$ for all $x \in \mE_s$. 
Then $f(x) = \tilde f(s) = \tilde f(v_D(x))$ for all $x \in \mE_s$, for all $s \in [0,M)$.  
Also, $f(x_*) = \tilde f(M)$. 
This proves that $f = \tilde f \circ v_D$ in the closure of $D$. 
The regularity of $\tilde f$ comes from that of $f$ 
and from the properties of the torsion function $v_D$ already recalled. 
One has $\tilde f(0) = 0$ because $f=0$ on the boundary $\pa D$, 
which is the level set $\mE_0$.

Now we write identity \eqref{eq:first.reduction.i} in terms of $\tilde f$ and $v_D$. 
First, by \eqref{grad.g1.J2.tilde.u} and \eqref{g.1.v.D}, one has 
$\mJ_2 \tilde u = \nabla g_1 = \alpha_0 \grad v_D$, whence 
$|\tilde u|^2 = \alpha_0^2 |\grad v_D|^2$. 
Second, by \eqref{g1.def} and \eqref{g.1.v.D}, one has 
$g = g_1 + \frac{\alpha_0}{4} |x|^2 
= k + \alpha_0 v_D + \frac{\alpha_0}{4} |x|^2$. 
Third, to write $H(f)$, we use the chain rule $\grad f = \tilde f'(v_D) \grad v_D$. 
Plugging these identities into \eqref{eq:first.reduction.i} gives \eqref{eq:first.reduction1}
with $c = c_1 - \alpha_0 k$.
\end{proof}

Now we show that the gradient of the torsion function 
on the boundary of a $C^2$ set
admits a lower bound. 
The proof is rather standard and it is based on the maximum principle.

\begin{lemma}\label{lem:lower.bound}
Let $D\subset \R^2$ be a bounded open set of class $C^{2}$, 
and let $v_D$ be its torsion function.
Then $v_D\in C^1(\overline D)$ and there exist two constants $C_0, C_1 \in \R$ such that 
\[
0 < C_0 \leq \inf_{\pa D} |\nabla v_D| 
\leq \sup_{\pa D} |\grad v_D| \leq C_1 < \infty. 
\]
\end{lemma}

\begin{proof}
The regularity of $v_D$ can be obtained by the classical regularity theory of elliptic operators, 
and the existence of the upper bound $C_1$ is a direct consequence of it. 
We prove the lower bound $C_0$.   

By the regularity and compactness of the boundary $\pa D$, 
there exists $r_0 > 0$ such that 
for every point $x_0 \in \pa D$ there exists a point $x_1 \in D$ such that 
the open disc $B_{r_0}(x_1)$ is contained in the open set $D$,  
the circle $\pa B_{r_0}(x_1)$ and the boundary $\pa D$ have one point in common, 
which is $x_0$, and no other common points, 
and $\pa B_{r_0}(x_1)$ is tangent to $\pa D$ at $x_0$. 

Let $x_0\in \pa D$, and let $B := B_{r_0}(x_1)$ be the tangent disc described above.  
The torsion function $v_B$ of the ball $B$ is explicit, and it is 
\[
v_B(x) = \frac{r_0^2 -|x-x_1|^2}{4}. 
\]
The difference $u := v_D-v_B$ satisfies
\[
\begin{cases}
\Delta u=0 & \text{in } B, 
\\
u \geq 0 & \text{on } \pa B, 
\end{cases}
\]
because $v_D \geq 0$ in $\overline D$, and, in particular, on $\pa B$, while $v_B = 0$ on $\pa B$. 
Also, $u(x_0) = 0$ because $x_0 \in \pa D \cap \pa B$. 
Since $u$ is harmonic in $B$, by the maximum principle we have 
\[
0 = u(x_0) = \min_{\pa B} u = \min_B u.
\]
Moreover $u$ is not constant in $B$ because $u>0$ on $(\pa B) \cap D = (\pa B) \setminus \{ x_0 \}$. 
Since $x_0$ is a minimum point for the harmonic function $u$, 
by Hopf Lemma we have
\[
\frac{\pa}{\pa \nu} u(x_0) < 0.
\] 
Since $\pa B$ is tangent to $\pa D$ from inside, 
the two curves have the same normal  
$\nu_{\pa D} (x_0) = \nu_{\pa B}(x_0)$ at $x_0$. Therefore
 \[
 \la \nabla v_D(x_0), \nu_{\pa D}(x_0)\ra
 - \la \nabla v_B(x_0), \nu_{\pa B}(x_0)
 = \la \nabla u(x_0), \nu_{\pa B}(x_0)\rangle <0.
 \]
By \eqref{torsion.gradient}, we have 
$\nabla v_D(x_0)= -|\nabla v_D(x_0)| \nu_{\pa D}(x_0)$ and 
$\nabla v_B(x_0)= -|\nabla v_B(x_0)| \nu_{\pa B}(x_0)$. 
Hence
\[
|\nabla v_D(x_0)| > |\nabla v_B (x_0)|= \frac{r_0}{2}.
\]
This proves the lower bound with $C_0 = r_0/2$, 
uniformly in $x_0 \in \pa D$. 
\end{proof}

We are now in position to prove the first part of 
Theorem \ref{thm:4.1}.
We compute
\begin{align*}
& \div \bigg( \frac{\tilde f'  (v_D) \nabla v_D }{\sqrt{ 1 + (\tilde f' (v_D))^2 |\nabla v_D|^2 }} \bigg)
\\ 
& = \frac{\tilde f'' (v_D) |\nabla v_D|^2 - \tilde f'(v_D) }{\sqrt{ 1 + (\tilde f'(v_D))^2 |\nabla v_D|^2}}
-\frac{ (\tilde f'(v_D))^2 \tilde f''(v_D) |\nabla v_D|^4 + \frac12 (\tilde f'(v_D))^3 
\langle \nabla ( |\nabla v_D|^2 ), \nabla v_D \rangle }
{ \{ 1 + (\tilde f'(v_D))^2 |\nabla v_D|^2 \}^\frac32}
\\
& = \frac{ \tilde f'' (v_D) |\nabla v_D|^2 - \tilde f' (v_D) \{ 1 + (\tilde f'(v_D))^2 |\nabla v_D|^2 \} 
- \frac12(\tilde f' (v_D))^3 \langle \nabla (|\nabla v_D|^2), \nabla v_D \rangle }
{ \{ 1 + (\tilde f'(v_D))^2 |\nabla v_D|^2 \}^\frac32}.
\end{align*}
As already observed above, 
by the maximum principle, the function $v_D$ is positive and by classic PDE results it is smooth. 
Moreover, the only critical point of $v_D$ is at the maximum of $v_D$, $v_D(x_*) = M$, 
since the torsion function of a convex domain is $1/2-$concave.
Therefore, every $s \in (0, M)$ is a regular value, i.e., 
for every $s$, on the level set $ \mE_s = \{v_D=s\}$ the gradient of $v_D$ does not vanish, 
the curve $\mE_s$ is smooth (by implicit function theorem) and connected, since $\sqrt {v_D}$ is concave. 
The curve $\mE_s$ defined in \eqref{def.mE.s} is a level set of the function $v_D$,
and therefore, by \eqref{curvatura.level.set}, 
its curvature is
\[
H_{\mE_s} = - \div \left(\frac{\nabla v_D}{|\nabla v_D|}\right)
= - \Big( \frac{\Delta v_D}{|\nabla v_D|} 
- \frac{\langle  \nabla (|\nabla v_D|), \nabla v_D\rangle}{|\nabla v_D|^2} \Big)
= \frac{|\nabla v_D| + \langle  \nabla (|\nabla v_D|), \nabla v_D\rangle}{|\nabla v_D|^2}.
\] 
Hence
\[
\frac12 \langle \nabla (|\nabla v_D|^2 ) , \nabla v_D\rangle 
= |\nabla v_D| \langle \nabla ( |\nabla v_D| ) , \nabla v_D\rangle 
= |\nabla v_D|^3 H_{\mE_s} - |\nabla v_D|^2,
\]
whence 
\[
\div \bigg( \frac{\tilde f'(v_D) \nabla v_D }{ \sqrt{ 1 + (\tilde f'(v_D))^2 |\nabla v_D|^2} } \bigg ) 
= \frac{ \tilde f''(v_D) |\nabla v_D|^2 - \tilde f'(v_D) - (\tilde f'(v_D))^3 |\nabla v_D|^3 H_{\mE_s} }
{ \{ 1 + (\tilde f'(v_D))^2 |\nabla v_D|^2 \}^\frac32 }.
\]
Thus identity \eqref{eq:first.reduction1} on the level set $\mE_s$, $s \in (0,M)$, 
becomes
\begin{equation} \label{eq:on.level.set}
\frac{\alpha_0^2}{2} |\nabla v_D|^2 
+ \alpha_0^2 s 
- \sigma_0 \frac{ \tilde f''(s) |\nabla v_D|^2 - \tilde f'(s) - (\tilde f'(s))^3 |\nabla v_D|^3 H_{\mE_s} }
{ \{ 1 + (\tilde f'(s))^2 |\nabla v_D|^2 \}^\frac32 }
= c \quad \text{on } \mE_s.
\end{equation} 

Now we prove some limit properties for the function $\tilde f$ close to the level $s=0$, 
related to properties of the profile function $f$ 
close to $\pa \Om \cap \{ x_3 = 0 \} = \pa D \times \{ 0 \}$. 

\begin{lemma}
The function $\tilde f$ obtained in Lemma \ref{lemma:tilde.f} satisfies 
\begin{equation} \label{lim.tilde.f.der.infty}
\lim_{s \to 0^+} \tilde f'(s) = \infty
\end{equation}
and 
\begin{equation}\label{eq:defn.beta}
\beta := \lim_{s \to 0^+} \frac{ \tilde f''(s) }{ (\tilde f'(s))^3 } \in \R.
\end{equation}
\end{lemma}

\begin{proof}
By Lemma \ref{lem:lower.bound}, $|\grad v_D|$ is bounded from below and above on $\pa D$, 
and therefore, by continuity, also on a neighborhood of $\pa D$ in $D$.
Consider a point $x_0 \in \pa D$. Let $\ph$ be the solution of the Cauchy problem 
\[
\ph'(t) = \frac{ (\grad v_D)(\ph(t)) }{ |(\grad v_D)(\ph(t))|^2 }, \quad \ 
\ph(0) = x_0
\]
in the interval $t \in [0, \delta]$, where $\delta > 0$ is sufficiently small 
to have $\ph(t) \in D$ for all $t \in (0, \delta]$. 
The composition $\psi(t) := v_D(\ph(t))$ satisfies 
$\psi(0) = v_D(x_0) = 0$ 
and $\psi'(t) = \la (\grad v_D)(\ph(t)) , \ph'(t) \ra = 1$. 
Hence $\psi(t) = t$ on $[0,\delta]$, namely  
\[
v_D(\ph(t)) = t \quad \ \forall t \in [0,\delta],
\]
so that $\ph(t)$ belongs to the level set $\mE_t$, for all $t \in [0, \delta]$.  
The point $(\ph(t), f(\ph(t)))$ on the boundary $\pa \Om$ 
converges, as $t \to 0$,
to the point $(\ph(0), f(\ph(0))) = (x_0, 0) \in \pa \Om$.
Hence the normal to the boundary $\pa \Om$ at the point $(\ph(t), f(\ph(t)))$,
which is given in \eqref{normal.pa.Om.f}, converges, as $t \to 0$, 
to the normal to $\pa \Om$ at the point $(x_0,0)$, which is given in \eqref{normal.at.pa.D}. 
In particular, the third component of these normal vectors satisfies 
\[
\lim_{t \to 0} \frac{1}{\sqrt{1 + |(\grad f)(\ph(t))|^2}} = 0.
\]
Hence $|(\grad f)(\ph(t))| \to \infty$ as $t \to 0$. 
Since $f = \tilde f \circ v_D$, one has $\grad f(x) = \tilde f'(v_D(x)) \grad v_D(x)$, 
and 
\[
(\grad f)(\ph(t)) = \tilde f'(v_D(\ph(t))) (\grad v_D)(\ph(t))  
= \tilde f'(t) (\grad v_D)(\ph(t)).
\]  
Since $|\grad v_D|$ is bounded, $|\tilde f'(t)| \to \infty$ as $t \to 0$. 
Hence $|\tilde f'| > 1$ in some interval $(0, \delta_1)$, 
and, since $\tilde f'$ is continuous, 
it cannot change sign in $(0, \delta_1)$.   
Moreover $\tilde f(0) = 0$, $\tilde f > 0$ on $(0, \delta_1)$, 
whence $\tilde f' > 0$ at a point, and therefore at all points of $(0, \delta_1)$.  
Hence $|\tilde f'| = \tilde f'$, which gives \eqref{lim.tilde.f.der.infty}.
Identity \eqref{eq:on.level.set} at the point $\ph(s)$, with $s \in (0, \delta]$, gives 
\[
\frac{ \tilde f''(s) }{ (\tilde f'(s))^3 } 
= \frac{ V(s) \mu(s) }{ \sigma_0 }
\Big\{ \frac{\alpha_0^2}{2} V^2(s) 
+ \alpha_0^2 s 
+ \frac{ \sigma_0 }{ \mu(s) (\tilde f'(s))^2 V^3(s) }
+ \frac{ \sigma_0 H_{\mE_s}(\ph(s)) }{ \mu(s) } 
- c \Big\},
\]
where 
\[
V(s) := |(\grad v_D)(\ph(s))|, \quad \  
\mu(s) := \Big( 1 + \frac{1}{(\tilde f'(s))^2 V^2(s)} \Big)^{\frac32}.
\]
For $s \to 0$, one has 
\[
\tilde f'(s) \to \infty, \quad \ 
V(s) \to |\grad v_D(x_0)|, \quad \  
\mu(s) \to 1, \quad \ 
H_{\mE_s}(\ph(s)) \to H_{\pa D}(x_0),
\]
whence 
\begin{equation} \label{beta.direct}
\lim_{s \to 0^+} \frac{ \tilde f''(s) }{ (\tilde f'(s))^3 } 
= \frac{ |\grad v_D(x_0)| }{ \sigma_0 }
\Big\{ \frac{\alpha_0^2}{2} |\grad v_D(x_0)|^2 
+ \sigma_0 H_{\pa D}(x_0) 
- c \Big\},
\end{equation}
which is a finite limit. 
\end{proof}

Taking the limit as $s \to 0^+$ in \eqref{eq:on.level.set}, 
or using directly formula \eqref{beta.direct} of the constant $\beta$,
we obtain \eqref{eq:over.determined}. 
To complete the proof of Theorem \ref{thm:4.1}, 
it remains to prove that the constants $\beta$ and $c$ depend only on the set $D$.

\subsection{The constants $\beta$ and $c$}\label{sub:constant}
The problem in \eqref{eq:over.determined} is related to the famous Serrin overdetermined problem. We now investigate further this link, by computing the values of the constants $\beta$ and $c$.
The definition of the constant $\beta$ in \eqref{eq:defn.beta} in terms of the profile function $f$ 
suggests that $\beta$ depends on the three-dimensional nature of the problem.
We show that this is not (completely) true, because,  
assuming that the drop $\Omega$ is smooth, the tangent plane to the boundary $\pa \Om$ 
at all points of $\pa \Om \cap \{ x_3 = 0 \}$ is vertical, 
the normal at those points is horizontal, 
and, thanks to this geometrical observation,  
we show that the knowledge of the planar set $D$ is sufficient to compute the value of $\beta$.
In the following lemma we compute the value of the constant $c$.
Recall that we are assuming that the barycenter of $\Omega$ is at the origin of $\R^3$, 
hence the origin of $\R^2$ is in $D$.
Also recall that $P(D)$ denotes the perimeter of the set $D$. 

\begin{lemma}\label{lem:constant.c}
The constant $c$ in \eqref{eq:first.reduction1} is uniquely determined by the formula
\begin{equation}\label{eq:constant.c}
c=\frac{1}{|D|} \left(\frac{3\alpha^2_0}{2} \int_{D} |\nabla v_D|^2\,dx +\sigma_0 P(D)  \right).
\end{equation}
\end{lemma}

\begin{proof} 
For $s \in (0, M]$, we consider the superlevel set 
$D_s := \{ x \in D : v_D(x) > s \}$, 
and we integrate \eqref{eq:first.reduction1} over $D_s$,
\[
\frac{ \alpha_0^2 }{2} \int_{D_s} |\nabla v_D|^2 \, dx
+ \alpha_0^2 \int_{D_s} v_D \, dx
- \sigma_0 \int_{D_s} \div \bigg( \frac{\tilde f'(v_D) \nabla v_D}
{\sqrt{1 + |\tilde f'(v_D) \nabla v_D|^2}} \bigg) \,dx
= c|D_s|.
\]
The boundary $\pa D_s$ is the level set $\mE_s$ in \eqref{def.mE.s} of the torsion function $v_D$, 
the gradient $\grad v_D$ is orthogonal to the curve $\mE_s$, 
and the outward unit normal to $\mE_s$ is 
\[
\nu_{\pa D_s} = - \frac{\grad v_D}{|\grad v_D|} \quad \ \text{on } \pa D_s.
\]
From the divergence theorem, 
\begin{align*}
\int_{D_s} \div \bigg( \frac{\tilde f'(v_D) \nabla v_D}
{\sqrt{1 + |\tilde f'(v_D) \nabla v_D|^2}} \bigg) \,dx 
& = \int_{\pa D_s} \la \frac{\tilde f'(v_D) \nabla v_D}{\sqrt{1 + |\tilde f'(v_D) \nabla v_D|^2}} , 
\nu_{\pa D_s} \ra \, d\sigma 
\\ 
& = - \int_{\pa D_s} \frac{\tilde f'(s) |\nabla v_D|}{\sqrt{1 + |\tilde f'(s) \nabla v_D|^2}} \, d\sigma. 
\end{align*}
By \eqref{lim.tilde.f.der.infty} and Lemma \ref{lem:lower.bound}, 
the function in the last integral converges to 1 uniformly as $s \to 0^+$,
and therefore the integral converges to $\int_{\pa D} 1 \, d\sigma = P(D)$.  
Taking the limit as $s \to 0^+$, we obtain 
\[
\frac{ \alpha_0^2 }{2} \int_{D} |\nabla v_D|^2 \, dx
+ \alpha_0^2 \int_{D} v_D \, dx
+ \sigma_0 P(D) 
= c|D|.
\]
By \eqref{torsion.integral}, we obtain \eqref{eq:constant.c}.
\end{proof}

As observed in Lemma \ref{lem:constant.c}, the constant $c$ can be computed just from the knowledge of $D$.
At this point, we are also able to compute the value of the constant $\beta$, 
showing that also $\beta$ is given only in terms of quantities related to the planar set $D$.
This completes the proof of Theorem \ref{thm:4.1}. 

\begin{lemma}\label{lem:beta}
Assume the hypotheses of Theorem \ref{thm:4.1}. 
The constant $\beta$ defined in \eqref{eq:defn.beta} is given by 
\begin{equation} \label{formula.in.lem:beta}
\beta = - \Big( \sigma_0 \int_{\pa D} \frac{\la x,\nu_{\pa D}\ra}{|\nabla v_D|} \, d\sigma \Big)^{-1} 
\Big\{ \sigma_0 P(D) + \alpha_0^2 \int_D |\nabla v_D|^2 \, dx \Big\}.
\end{equation}
\end{lemma}

\begin{proof} 
Multiplying \eqref{eq:over.determined} by $\la x,\nu_{\pa D}\ra$ 
and integrating along the curve $\pa D$, we get 
\begin{equation} \label{along.01}
\frac{\alpha_0^2}{2} \int_{\pa D} | \nabla v_D|^{2} \la x, \nu \ra \, d\sigma 
- \sigma_0 \beta \int_{\pa D} \frac{ \la x, \nu \ra }{ | \nabla v_D| } \, d\sigma
+ \sigma_0 \int_{\pa D} H_{\pa D} \la x, \nu \ra \, d\sigma
= c \int_{\pa D} \la x, \nu \ra \, d\sigma.
\end{equation}
We calculate each term. 
By the divergence theorem,
\begin{align*}
\int_{\pa D} |\nabla v_D|^2 \la x,\nu_{\pa D} \ra \, d\sigma
& = \int_D \div \{ x |\nabla v_D|^2 \} \, dx 
\\ 
& = \int_D (\div x) |\nabla v_D|^2\, dx 
+ \int_D \la x, \nabla ( |\nabla v_D|^2 ) \ra \, dx
\\ 
& = 2 \int_D |\nabla v_D |^2 \, dx 
+ 2 \int_D \la (\nabla^2 v_D) (\nabla v_D) , x \ra \, dx,
\end{align*}
where $\grad^2 v_D$ is the Hessian matrix of $v_D$. 
From the general identity
\[
\div \{ \la \grad w , x \ra \grad w \} 
= \la \grad w , x \ra \Delta w 
+ |\grad w|^2 
+ \la (\grad^2 w)(\grad w), x \ra
\]
applied to the torsion function $v_D$, using the identity $\Delta w_D = -1$, we find 
\[
\int_D \la (\nabla^2 v_D) (\nabla v_D), x\ra\,dx
= \int_D \div \{ \la \nabla v_D, x \ra \nabla v_D \} \, dx 
+ \int_D \la \nabla v_D, x \ra \, dx 
- \int_D |\nabla v_D|^2 \, dx.
\]
By the divergence theorem and formula \eqref{torsion.gradient} for the gradient of $v_D$,  
\[
\int_D \div \{ \la \nabla v_D, x \ra \nabla v_D \} \, dx 
= \int_{\pa D} \la \nabla v_D, x \ra \la \nabla v_D , \nu_{\pa D} \ra \, d\sigma
= \int_{\pa D} |\nabla v_D|^2 \la x,\nu_{\pa D}\ra \, d\sigma.
\]
By the divergence theorem applied to the vector field $x v_D$, 
which has divergence $\div(x v_D) = (\div x) v_D + \la x, \grad v_D \ra 
= 2 v_D + \la x, \grad v_D \ra$ in $D$ 
and boundary value $x v_D = 0$ on $\pa D$ by \eqref{eq:torsion},    
one has 
\[
0 = \int_{\pa D} \la x v_D , \nu_{\pa D} \ra \, d\sigma 
= \int_D \div(x v_D) \, dx 
= 2 \int_D v_D \, dx + \int_D \la x , \grad v_D \ra \, dx,
\]
and, by \eqref{torsion.integral},
\[
\int_D \la x , \grad v_D \ra \, dx 
= - 2 \int_D v_D \, dx 
= - 2 \int_D |\grad v_D|^2 \, dx.
\]
From these identities, we obtain  
\begin{equation}\label{eq:boundary.interior}
\int_{\pa D}|\nabla v_D|^2\la x,\nu_{\pa D}\ra \, d\sigma 
= 4 \int_{D}|\nabla v_D|^2\,dx.
\end{equation}
By \eqref{eq:tang.divergence} with $X(x)=x$, 
since $\div_{\pa D} x=1$, we get
\[
\int_{\pa D} H_{\pa D} \la x, \nu_{\pa D} \ra \, d\sigma = P(D), 
\]
and, from the divergence theorem, 
\[
\int_{\pa D} \la x, \nu_{\pa D} \ra \, d\sigma 
= \int_D \div x \, dx 
= \int_D 2 \, dx 
= 2 |D|.
\]
Hence \eqref{along.01} becomes 
\[
2 \alpha_0^2 \int_{D}|\nabla v_D|^2\,dx 
- \sigma_0 \beta \int_{\pa D} \frac{\la x, \nu_{\pa D} \ra }{ |\grad v_D| } \, d\sigma 
+ \sigma_0 P(D) = 2 c |D|.
\]
Replacing the constant $c$ with its formula in \eqref{eq:constant.c}, 
we get the result. 
\end{proof}

There is also another way to determine the constants $\beta$ and $c$. 
\begin{lemma}
Assume the hypotheses of Theorem \ref{thm:4.1}. 
Then the constants $\beta$ and $c$ satisfy 
\begin{align} 
\sigma_0 \beta \Big( \int_{\pa D} \frac{1}{|\nabla v_D|} \, d\sigma \Big) 
+ c P(D) 
& = \frac{\alpha_0^2}{2} \Big( \int_{\pa D} |\nabla v_D|^{2} \, d\sigma \Big) 
+ 2 \pi \sigma_0,
\label{eq:syst.beta.c.i} 
\\
\sigma_0 \beta P(D) 
+ c|D| 
& = \frac{\alpha_0^2}{2} \Big( \int_{\pa D} |\nabla v_D|^{3}\, d\sigma \Big) 
+ \sigma_0 \Big( \int_{\pa D} H_{\pa D} |\nabla v_D| \, d\sigma \Big).
\label{eq:syst.beta.c.ii}
\end{align}
Moreover, the solution of system \eqref{eq:syst.beta.c.i}, \eqref{eq:syst.beta.c.ii} 
is unique if and only if $D$ is not a disk.
\end{lemma}

\begin{proof}
The constants $c$ and $\beta$ are uniquely determined if $D$ is not a disk.
In fact, integrating \eqref{eq:over.determined} along $\pa D$, 
by \eqref{Gauss.Bonnet} we get
\[
\frac{\alpha_0^2}{2}\int_{\pa D} |\nabla v_D|^{2} \, d\sigma 
- \sigma_0 \beta \int_{\pa D} \frac{1}{ |\nabla v_D| }\, d\sigma  
+ 2 \pi \sigma_0 = c P(D).
\]
Moreover, multiplying \eqref{eq:over.determined} by $|\nabla v_D|$ 
and integrating along $\pa D$, by \eqref{torsion.gradient.integral} we have
\[
\frac{\alpha_0^2}{2}\int_{\pa D}  | \nabla v_D|^{3}\, d\sigma -\sigma_0\beta P(D) + \sigma_0 \int_{\pa D}H_{\pa D}|\nabla v_D|\, d\sigma= c\int_{\pa D} |\nabla v_D|\,d\s=c|D|.
\]
Therefore $\beta$ and $c$ satisfy \eqref{eq:syst.beta.c.i}, \eqref{eq:syst.beta.c.ii}.
This linear system has a unique solution if and only if the corresponding matrix has nonzero determinant, 
i.e., if and only if the difference  
\[
\Big( \int_{\pa D} \frac{1}{|\grad v_D|} \, d\sigma \Big) |D| - P^2(D)
\]
is nonzero.  
By H\"older inequality, 
\[
P(D) = \int_{\pa D} 1 \, d\sigma 
= \int_{\pa D} \frac{ |\grad v_D|^{\frac12} }{ |\grad v_D|^{\frac12} } \, d\sigma 
\leq 
\Big( \int_{\pa D} |\grad v_D| \, d\sigma \Big)^{\frac12} 
\Big( \int_{\pa D} \frac{1}{|\grad v_D|} \, d\sigma \Big)^{\frac12},
\]
with equality if and only if $|\nabla v_D|$ is constant, 
i.e., by Serrin theorem, if and only if $D$ is a disk. 
\end{proof}

\begin{lemma} \label{teo:consequence}
Assume the hypotheses of Theorem \ref{thm:4.1}. Then 
\begin{align} 
\sigma_0 \beta  
& = \frac{1}{P(D)} \Big\{ 
\frac{\alpha_0^2}{2} 
\Big( \int_{\pa D} |\nabla v_D|^{3}\, d\sigma - 3 \int_{D} |\nabla v_D|^2\,dx \Big) 
\notag \\ 
& \quad \ 
+ \sigma_0 \Big( \int_{\pa D} H_{\pa D} |\nabla v_D| \, d\sigma - P(D) \Big) \Big\}.
\label{resa.04}  
\end{align}
\end{lemma}

\begin{proof}
Plug \eqref{eq:constant.c} into \eqref{eq:syst.beta.c.ii}.
\end{proof}

\subsection{Torsion function and rigidity} 
\label{sec:global}

We now study in details the constant $\beta$. 
This will be the key to conclude our proof.

By \eqref{formula.in.lem:beta}, 
we see immediately that $\beta<0$, 
because the scalar product $\la x,\nu_{\pa D}\ra$ is positive 
for all $x \in \pa D$; in fact, this property holds for all star-shaped sets 
with respect to the origin.
To obtain a stronger bound on $\beta$, 
we use \eqref{resa.04}, starting with analysing the terms with coefficient $\sigma_0$. 
We use Reilly formula \eqref{Reilly.in.sec.2}.

\begin{lemma}\label{lem:reilly}
Let $D \subset \R^2$ be a bounded open convex set with $C^2$ boundary, 
and let $v_D$ be its torsion function. Then
\begin{equation}\label{eq:grad.curvature}
\int_{\pa D} H_{\pa D} |\nabla v_D|\, d\sigma \leq \frac{P(D)}{2}.
\end{equation}
\end{lemma}

\begin{proof} 
To lighten the notation, in this proof we denote $v$ the torsion function $v_D$ of the set $D$, 
$\nu$ the normal vector $\nu_{\pa D}$, 
and $H$ the mean curvature $H_{\pa D}$. 
Reilly formula \eqref{Reilly.in.sec.2} applied to the set $D$ 
and the torsion function $v$ is
\[
\int_{\pa D} H \la \grad v, \nu \ra^2\, d\sigma 
= \int_D \Big( (\Delta v)^2 - |\nabla^2 v|^2 \Big) \, dx,
\]
where $\grad^2 v$ is the Hessian matrix of $v$. 
By Cauchy-Schwarz inequality, 
one has the pointwise inequality for the Hessian matrix and its trace 
\[
\frac{(\Delta v)^2}{2} \leq	 |\nabla^2 v|^2.
\]
Therefore, by \eqref{torsion.gradient},
\[
\int_{\pa D} H |\nabla v|^2\, d\sigma 
= \int_{\pa D} H \la \grad v , \nu \ra^2 \, d\sigma 
\leq \frac12 \int_D (\Delta v)^2\, dx
= \frac{|D|}{2}.
\]
Since $D$ is convex, $H \geq 0$.
Hence, by H\"older inequality, Gauss-Bonnet formula \eqref{Gauss.Bonnet}, 
and the isoperimetric inequality, we have
\[
\int_{\pa D} H |\nabla v|\, d\sigma
\leq \Big( \int_{\pa D} H |\nabla v|^2\, d\sigma \Big)^\frac12 
\Big( \int_{\pa D} H \, d\sigma \Big)^\frac12  
\leq \sqrt{\pi |D|}
\leq \frac{P(D)}{2}. 
\qedhere
\]
\end{proof}

Now we study the terms with coefficient $\alpha_0^2$ in \eqref{resa.04}. 

\begin{lemma} \label{lem:det.01}
Let $D, v_D$ be like in Lemma \ref{lem:reilly}. Then 
\begin{gather} \label{int.det.new}
2 \int_D v_D \det(\grad^2 v_D) \, dx
= \int_D |\grad v_D|^2 \, dx
+ \int_D \la (\grad^2 v_D) \grad v_D , \grad v_D \ra \, dx,
\\
\label{eq:determinant}
\int_{\pa D} |\nabla v_D|^3\, d\sigma 
= 3\int_{D} |\nabla v_D|^2\,dx - 4 \int_{D} v_D \det (\nabla^2 v_D) \,dx,
\end{gather}
where $\nabla^2 v_D$ is the Hessian matrix of $v_D$. 
\end{lemma}

\begin{proof} 
To lighten the notation, in this proof we denote $v$ the torsion function $v_D$, 
$v_j = \pa_{x_j} v$, $v_{jk} = \pa_{x_j x_k} v$ its partial derivatives,
and $\nu$ the normal $\nu_{\pa D}$.
By \eqref{torsion.gradient} and the divergence theorem, 
\begin{align}
\int_{\pa D} |\nabla v|^3\, d\sigma 
& = - \int_{\pa D} |\grad v|^2 \la \grad v, \nu \ra \, d\sigma
\notag \\ 
& = - \int_D \div \{ |\grad v|^2 \grad v \} \,dx
\notag \\ & 
= - \int_D |\grad v|^2 \Delta v \, dx 
- 2 \int_D \la (\grad^2 v) \grad v , \grad v \ra \, dx.
\label{int.3.2}
\end{align}
One has 
\[
\int_D |\grad v|^2 \Delta v \, dx 
= - \int_D |\grad v|^2 \, dx 
\]
because $\Delta v = -1$ in $D$. 
For the other integral, we apply the divergence theorem to the vector field 
$F = (\grad^2 v)(\grad v) v$, and we get 
\[
\int_D \div \{ (\grad^2 v)(\grad v) v \} \, dx 
= \int_D \div F \, dx 
= \int_{\pa D} \la F, \nu \ra \, d\sigma 
= 0
\]
because $v = 0$ on $\pa D$, whence also $F = 0$ on $\pa D$.
We calculate 
\begin{align*}
\div \{ (\grad^2 v)(\grad v) v \} 
& = \div \{ (\grad^2 v)(\grad v) \} v + \la (\grad^2 v)(\grad v) , \grad v \ra
\end{align*}
and 
\begin{align*}
\div \{ (\grad^2 v)(\grad v) \} 
= \sum_{k,j} \pa_{x_k} (v_{kj} v_j) 
= \sum_{k,j} v_{kkj} v_j + \sum_{k,j} v_{kj}^2
= \la \grad ( \Delta v) , \grad v \ra + |\grad^2 v|^2 
= |\grad^2 v|^2, 
\end{align*}
because $\grad (\Delta v) = \grad (-1) = 0$. Moreover, 
\[
|\grad^2 v|^2 
= v_{11}^2 + 2 v_{12}^2 + v_{22}^2
= (v_{11} + v_{22})^2 - 2 v_{11} v_{22} + 2 v_{12}^2
= (\Delta v)^2 - 2 \det(\grad^2 v).
\]
Since $(\Delta v)^2 = 1$, we obtain 
\begin{align*}
\int_D \la (\grad^2 v) \grad v , \grad v \ra \, dx
& = - \int_D v \div \{ (\grad^2 v)(\grad v) \} \, dx 
\\ 
& = - \int_D v |\grad^2 v|^2 \, dx 
\\ 
& = - \int_D v \{ 1 - 2 \det(\grad^2 v) \} \, dx 
\\ 
& = - \int_D v \, dx + 2 \int_D v \det(\grad^2 v) \, dx, 
\end{align*}
and, recalling \eqref{torsion.integral}, we get \eqref{int.det.new}.
Then \eqref{eq:determinant} trivially follows from \eqref{int.3.2} and \eqref{int.det.new}.
\end{proof}

When $D$ is an ellipse, 
the torsion function $v_D$ is concave (see \cite{HNST}), 
and therefore the eigenvalues of the Hessian matrix $\nabla^2 v_D$ are both $\leq 0$; 
this implies that $\det(\nabla^2 v_D) \geq 0$, and hence, by \eqref{eq:determinant}, 
\begin{equation}\label{eq:upper.bound}
\int_{\pa D}|\nabla v_D|^3\, d\sigma \leq 3\int_{D}|\nabla v_D|^2\, d\sigma.
\end{equation}

For a more general convex set $D$, the torsion function $v_D$ is $\frac12$-concave, 
that is, $\sqrt{ v_D }$ is concave, 
but, in general, $v_D$ itself is not concave, 
so that we cannot use the inequality $v_D \det(\nabla^2 v_D) \geq 0$ \emph{pointwise}. 
Nonetheless, in the next lemma we show that inequality \eqref{eq:upper.bound} holds for general convex sets.

\begin{lemma}\label{lem:inequality.for.gradient1}
Let $D, v_D$ be like in Lemma \ref{lem:reilly}. Then 
\begin{equation} \label{int.det.pos}
\int_D v_D \det(D^2 v_D) \, dx \geq 0,
\end{equation}
and \eqref{eq:upper.bound} holds. 
\end{lemma}

\begin{proof} 
To lighten the notation, in this proof we denote $v$ the torsion function $v_D$, 
$v_j = \pa_{x_j} v$, $v_{jk} = \pa_{x_j x_k} v$ its partial derivatives,
and $\nu$ the normal $\nu_{\pa D}$.
From \cite{ML}, we know that $v$ is $\frac12$-concave, that is, the Hessian matrix of $v^{\frac12}$ 
is a nonpositive definite matrix. 
Hence its determinant (product of its two eigenvalues) is $\geq 0$. 
We calculate the partial derivatives 
\[
\pa_{ij} (v^{\frac12}) 
= \frac{ 2 v v_{ij} - v_i v_j }{ 4 v^{3/2} },
\]
and 
\begin{align*}
0 \leq \det[ \Hess (v^{\frac12}) ] 
& = \frac{1}{16 v^3} \big\{ (2 v v_{11} - v_1^2)(2 v v_{22} - v_2^2) - (2 v v_{12} - v_1 v_2)^2 \big\} 
\\ 
& = \frac{\det (\Hess v)}{4v} 
- \frac{ v_1^2 v_{22} + v_2^2 v_{11} - 2 v_1 v_2 v_{12} }{8 v^2}.
\end{align*}
Multiplying by $8 v^2$, 
\[
2 v \det (\Hess v) \geq v_1^2 v_{22} + v_2^2 v_{11} - 2 v_1 v_2 v_{12}.
\]
Since $v_{11} + v_{22} = \Delta v = -1$, we have 
\begin{align*}
v_1^2 v_{22} + v_2^2 v_{11} - 2 v_1 v_2 v_{12}
& = v_1^2 (- 1 - v_{11}) + v_2^2 (- 1 - v_{22}) 
- 2 v_1 v_2 v_{12}
\\ 
& = - |\grad v|^2 - \la (\Hess v) \grad v , \grad v \ra.
\end{align*}
Integrating over $D$, and recalling \eqref{int.det.new}, we obtain
\[
2 \int_D v \det (\Hess v) \, dx 
\geq - \int_D |\grad v|^2 \, dx - \int_D \la (\Hess v) \grad v , \grad v \ra \, dx
= - 2 \int_D v \det (\Hess v) \, dx,
\]
which gives \eqref{int.det.pos}. 
Then \eqref{eq:upper.bound} follows trivially from \eqref{int.det.pos} and \eqref{eq:determinant}.
\end{proof}

Now we have all the ingredients to prove Theorem \ref{thm:global}.

\begin{proof}[\textbf{Proof of Theorem \ref{thm:global}.}]
From \eqref{resa.04}, \eqref{eq:grad.curvature}, \eqref{eq:upper.bound} 
(see Lemma \ref{lem:inequality.for.gradient1}), 
it follows that 
\begin{equation} \label{beta.1.2}
\beta \leq -\frac12.
\end{equation}
By Theorem \ref{thm:4.1}, the torsion function $v_D$ satisfies \eqref{eq:over.determined}, that is, 
\[
F_1(|\nabla v_D|) = c - \s_0 H_{\pa D} \quad \ \text{on } \pa D,
\]
where $F_1$ is the real function 
\[ 
F_1 : (0, \infty) \to \R, \quad \ 
F_1(t) := \frac{\alpha_0^2}{2} t^2 -\frac{\sigma_0\beta}{t} 
= \frac{\alpha_0^2}{2} t^2 + \frac{\sigma_0 |\beta|}{t}.
\]
The function $F_1$ is decreasing in the interval $(0, t_1]$, where 
\[
t_1 := \Big( \frac{\sigma_0 |\beta|}{\alpha_0^2} \Big)^{\frac13}.
\]
By \eqref{beta.1.2}, one has $|\beta| \geq \frac12$, whence
\[
t_1 \geq \Big( \frac{\sigma_0}{2\alpha_0^2} \Big)^{\frac13}.
\]
By Theorem \ref{teo:grad.estimate}, 
\[ 
\|\nabla v_D\|_{L^\infty(D)} < \sqrt{\frac{|D|}{2\pi}}.
\] 
Therefore, if 
\begin{equation} \label{cond.01}
\sqrt{\frac{|D|}{2\pi}} \leq t_1,
\end{equation}
then $|\nabla v_D| \in (0,t_1)$. 
Note that inequality \eqref{eq:sigma.alpha} implies \eqref{cond.01}. 
Since $F_1$ is invertible in $(0,t_1]$, 
with decreasing inverse $F_1^{-1}$, 
we have 
\begin{equation} \label{eq:non.increasing}
\frac{\pa v_D}{\pa \nu_{\pa D}} 
= \la \grad v_D , \nu_{\pa D} \ra 
= -|\nabla v_D|
= - F_1^{-1}(c-\sigma_0 H_{\pa D}) 
=: \mathrm f(H_{\pa D}),
\end{equation}
and $\mathrm f$ is decreasing.  
To apply Serrin Theorem \ref{thm:serrin} to the set $D$, 
we just need to show that its boundary $\pa D$ is of class $C^3$. 
This is achieved by a simple bootstrap argument based on Schauder regularity theory 
(see \cite{GT}, Chapter 6, and \cite{FR.Ros.Oton.2022}).
By assumption, the boundary $\pa \Omega$ of the fluid domain $\Omega \subset \R^3$ is of class $C^2$, 
hence the boundary $\pa D$ of the planar set $D \subset \R^2$ in \eqref{eq:omega.is.graph} is also of class $C^2$. 
Thus, in particular, $\pa D$ is of class $C^{1,\alpha}$ for all $\alpha \in (0,1)$, 
and therefore (see \cite{FR.Ros.Oton.2022}, page 69, at the end of section 2.6), 
the torsion function $v_D$, which is the solution of \eqref{eq:torsion},  
is $C^{1,\alpha}(\overline{D})$ for any $\alpha \in (0,1)$. 
Since $\nabla v_D$ is nonzero on $\pa D$, with lower and upper bound in Lemma \ref{lem:lower.bound}, 
from identity \eqref{eq:over.determined} we obtain that
\begin{equation} \label{H.0.alpha}
H_{\pa D} \in C^{0,\alpha}(\pa D).
\end{equation}
Since we already know that the curve $\pa D$ is of class $C^2$, 
\eqref{H.0.alpha} implies that $\pa D$ is of class $C^{2,\alpha}$.
The proof is classical (standard frozen coefficients argument for linear elliptic differential operators). 
Now, since $\pa D$ is $C^{2,\alpha}$, the torsion function $v_D$ is $C^{2,\alpha}(\overline D))$
(see Theorem 2.35 in \cite{FR.Ros.Oton.2022}, page 69). 
Hence $H_{\pa D} \in C^{1,\alpha}(\pa D)$, which implies that $\pa D$ is of class $C^{3,\alpha}$. 
Since $\pa D$ is $C^3$, we apply Theorem \ref{thm:serrin}, so that $D$ is a disk. 
Thus, Remark \ref{rem:lopez} applies, 
and the proof of Theorem \ref{thm:global} is complete.
\end{proof}

\begin{remark}\label{rmk:final}
Theorem \ref{thm:global} holds also assuming $\Omega$ convex, instead of strictly convex.
The only change is that $\pa \Omega$ may also contain a cylindrical part between the two graphical parts. 
The arguments in section \ref{sec:necessary.condition} and \ref{sec:global} 
that lead to prove that $D$ is a disk 
are still valid (the notation becomes a bit heavier but the proofs work). 
Finally, when computing the exact solution following Remark \ref{rem:lopez}, 
one can show that $\Omega$ is necessarily strictly convex. 
Indeed, $\pa \Omega$ would contain a cylindrical part 
if the mean curvature of $\pa \Omega$ at the curve joining the graphical part 
and the cylindrical part would not match, 
hence one would have a discontinuous mean curvature, 
in contradiction with the assumption that $\pa \Omega$ is of class $C^2$.
Thus, we use strictly convexity as a simplifying assumption, 
but in fact it is not necessary to show our rigidity result. 
\end{remark}
 
%

\bigskip

\begin{flushright}

\textbf{Pietro Baldi}

Dipartimento di Matematica e Applicazioni ``R. Caccioppoli''

University of Naples Federico II

Via Cintia, Monte Sant'Angelo, 80126 Naples, Italy

pietro.baldi@unina.it

\medskip

\textbf{Domenico Angelo La Manna}

Dipartimento di Matematica e Applicazioni ``R. Caccioppoli''

University of Naples Federico II

Via Cintia, Monte Sant'Angelo, 80126 Naples, Italy

domenicoangelo.lamanna@unina.it

\medskip

\textbf{Giuseppe La Scala}

Mathematical and Physical Sciences for Advanced Materials and Technologies

Scuola Superiore Meridionale

Via Mezzocannone, 4, 80138 Naples, Italy

giuseppe.lascala-ssm@unina.it

\end{flushright}

\end{document}